\documentclass[a4paper,11pt]{amsart}
\usepackage{comment}
\usepackage{amscd}
\usepackage{amsmath}
\usepackage{amssymb}
\usepackage{epsf,latexsym,graphicx,dsfont}
\usepackage[matrix,arrow,tips,curve]{xy}
\usepackage{color}
\usepackage{tikz}
\newlength{\defbaselineskip} 
\usepackage{footnote}
\usepackage{color}
\usepackage{todonotes}
\usepackage{MnSymbol}

\usepackage{amsfonts,amssymb,amscd}
\newtheorem{thm}{Theorem}[section]
\newtheorem{cor}[thm]{Corollary}
\newtheorem{corr}[thm]{Corollary}

\newtheorem{lem}[thm]{Lemma}
\newtheorem{prop}[thm]{Proposition}

\newtheorem{prob}[thm]{Problem}

\theoremstyle{definition}

\newtheorem{rem}[thm]{Remark}

\newtheorem{claim}[thm]{Claim}
\usepackage{fullpage}
\usepackage{tikz,tikz-cd}
\usetikzlibrary{matrix,arrows}
\makeatletter
\tikzset{
  edge node/.code={%
      \expandafter\def\expandafter\tikz@tonodes\expandafter{\tikz@tonodes #1}}}
\makeatother
\tikzset{
  subseteq/.style={
    draw=none,
    edge node={node [sloped, allow upside down, auto=false]{$\subseteq$}}},
  Subseteq/.style={
    draw=none,
    every to/.append style={
      edge node={node [sloped, allow upside down, auto=false]{$\subseteq$}}}
  }
}

 \numberwithin{equation}{section}
\numberwithin{equation}{section} \theoremstyle{definition}
\DeclareMathOperator{\Pic}{Pic}
 \DeclareMathOperator{\Spec}{Spec}

\DeclareMathOperator{\im}{im}

          \newcommand\PP{{\mathbb{P}}}
           \newcommand\QQ{{\mathbb{Q}}}

           \newcommand\U{{\mathcal U}}
          
          \newcommand\C{{\mathcal{C}  }}
          
            \newcommand\M{{\mathcal M}}
 \newcommand\F{{\mathcal F}}

\newcommand\E{\mathcal E}

          \newcommand\oo{\mathcal O}
          \newcommand\ZZ{\mathbb{Z}}

          \newcommand\rk{\mathrm{rk}}

          \newcommand\HK{hyper-K\"ahler}

\definecolor{zielony}{rgb}{0.5, 0.9, 0.1}
\definecolor{czerwony}{rgb}{0.8, 0.2, 0.1}
\definecolor{niebieski}{rgb}{0.3, 0.1, 0.9}

\newcounter{appendice}

\author[Grzegorz Kapustka]{Grzegorz Kapustka}
\address{G. Kapustka: Department of Mathematics and Informatics, Jagiellonian University, \L ojasiewicza 6, 30-348, Krak\'ow, Poland}
\email{grzegorz.kapustka@uj.edu.pl}
\author[Micha\l \ Kapustka]{Michal Kapustka}
\address{M. Kapustka: Institute of Mathematics of the Polish Academy of Sciences, ul. Śniadeckich 8, 00-656 Warszawa, Poland}
\email{michal.kapustka@impan.pl}
\begin{document}
\title{Constructions of derived equivalent hyper-K\"ahler fourfolds}

\begin{abstract}
    We study twisted  derived equivalences of \HK{} fourfolds. We describe when two \HK{} fourfolds of $K3^{[2]}$-type of Picard rank $1$ with isometric transcendental lattices 
    are derived equivalent. Then we present new constructions of pairs of twisted derived equivalent \HK{} manifolds of Picard rank $\geq 2$.
\end{abstract}
\subjclass{14J42, 14F08}
\keywords{
\HK{} fourfolds, derived categories, Fourier-Mukai partners}
\maketitle
\section{Introduction}
The study of equivalences of derived categories of smooth projective varieties goes back to the works of Mukai \cite{Mu,M2} who generalised the idea of Fourier transforms to abelian varieties.
 Finding Fourier-Mukai partners, i.e.~non-isomorphic derived equivalent varieties, has since become a central problem in algebraic geometry, closely related to homological mirror symmetry and the study of moduli spaces of vector bundles.

By results of Mukai and Orlov \cite{Mu}, \cite{O}, two K3 surfaces are derived equivalent if and only if their transcendental lattices are Hodge isometric.
In that case, one of the surfaces is a fine moduli space of stable sheaves on the second surface and the Fourier-Mukai kernel of the equivalence is given by the universal family. In a similar way, moduli spaces of twisted sheaves give rise to twisted derived equivalences of K3 surfaces, which induce rational Hodge isometries of transcendental lattices (see \cite{HS1}).
Conversely, if the transcendental lattices of two surfaces are Hodge isometric over $\QQ$ then there is a sequence of twisted derived equivalences connecting them (see \cite{H}).
For hyper-K\"ahler manifolds, that are natural generalisations of K3 surfaces to higher dimension, the situation is more complicated.

In this paper, we study derived equivalence of four dimensional \HK\ manifolds. 
More precisely, we study this problem in the case of one of the two known families of \HK\ fourfolds; the deformations of the Hilbert scheme of two points on a K3 surface called of $K3^{[2]}$-type \cite{B,F}.
Recall that strong restrictions in the classification of hyper-K\"ahler fourfolds where found in \cite{DHMV}.

It was proved in \cite{T} (see \cite[Corollary 9.3]{B}) that two derived equivalent $K3^{[n]}$-type manifolds have isometric transcendental lattices. In this paper we discuss the inverse problem. 
\begin{prob}\label{prob}
    Is it true that if two \HK{} manifolds have isometric transcendental lattices then they are derived equivalent?
\end{prob}
Recall that, by considering  equivalences of Hilbert schemes of points on $K3$ surfaces induced by equivalences between those $K3$ surfaces, non-isomorphic and derived equivalent $K3^{[n]}$-type manifolds with Picard group of rank $\geq 2$ were found in \cite{P,MMY}. Moreover, examples of twisted derived equivalent $K3^{[2]}$-type manifolds that are moduli spaces of sheaves were described in \cite{S,ADM}. In those cases  Problem \ref{prob} has a positive answer. Moreover, we shall see that $K3^{[2]}-$ type manifolds with Hodge isometric transcendental lattices are always twisted derived equivalent. 
 On the other hand, there exist $K3^{[n]}$-type manifolds (for each $n>2$) of Picard number two which have isometric transcendental lattice, but are not derived equivalent (cf. \cite{MM}).
 
In this paper, we discuss Problem \ref{prob} in the case of $K3^{[2]}$-type manifolds of Picard number one.
In particular, we prove in Theorem \ref{main1} below, that manifolds of Picard number one with isometric transcendental lattices are derived equivalent in the cases where their polarisation has suitable degree and divisibility. The remaining cases are treated in Theorem \ref{main2} and lead to twisted derived equivalences.  Building on the negative answer to Problem \ref{prob} for higher Picard number, we expect that our results are sharp in the sense that in cases treated in Theorem \ref{main2}, there should exist pairs of manifolds of Picard number one with polarisation of given degree and divisibility, which have isometric transcendental lattices, but are not derived equivalent.

Concretely, our main results are the following.
\begin{thm}\label{main1}
  Let $(X,L_X)$ and $(Y,L_Y)$ be two polarized \HK\ fourfolds of $K3^{[2]}$-type of Picard number one and degree $2d$ such that $d\equiv 1 \ \mod\ 4$ or $8|d$ or $L_X$ and $L_Y$ have divisibility 2. Then their transcendental lattices  $T_X$ and $T_Y$ are Hodge isometric if and only if $D^b(X)=D^b(Y)$. The number of Fourier-Mukai partners of $X$ is then  {
  $$\begin{cases}
      2^{\tau(d)} \text{ when }\rm{div}(L_X) =1,\\
      2^{\tau(d)-1} \text{ when }\rm{div}(L_X) =2,
  \end{cases}
  $$}  where $\tau(d)$ is the number of different prime divisors of $d$ and $\tau(1)=1$. 
\end{thm}

In the case of \HK{} manifolds of $K3^{[2]}$-type of Picard number one of other degrees, we deduce twisted derived equivalence, with explicit twist.
Denote by $D^b(X,[\frac{\delta_X}{2}])$ the derived category of twisted sheaves with twist that is the uniquely defined Brauer class with a $B$-lift $\frac{\delta_X}{2}$ such that $\delta_X$ is a $-2$ class of divisibility $2$ (see Section \ref{brauerz}).

\begin{thm}\label{main2}
    Let $(X,L_X)$ and $(Y,L_Y)$ be two polarized \HK\ fourfolds of $K3^{[2]}$-type of Picard number one and degree $2d$ such that $L_X$ and $L_Y$ are  divisibilty 1, $d$ is not divisible by 8 and  $d \not\equiv 1  \mod 4$ and the transcendental lattices $T_X$ and $T_Y$ are Hodge isometric. Then $D^b(X,[\frac{\delta_X}{2}])=D^b(Y,[\frac{\delta_Y}{2}])$. 
   
\end{thm}

The theorems are proved at the end of Section \ref{Sec4} as Theorem \ref{MAIN}.
The idea of the proof is to proceed via deformation of Fourier-Mukai kernels starting from special examples. First, having an isometry between transcendental lattices $T_X$ and $T_Y$ of two \HK{} fourfolds $X$, $Y$ of Picard rank 1 we extend it to a rational Hodge isometry between $H^2(X,\QQ)$ and $H^2(Y,\QQ)$. Next, we realise this rational isometry as a rational isometry given by a Fourier-Mukai kernel on $X'\times Y'$ 
of an equivalence of derived categories of Hilbert squares of K3 surfaces $X'$ and $Y'$. Then, following \cite[\S 5]{Mar1} (using hyperholomorphic sheaves \cite{V}), we deform the Fourier-Mukai kernel on $X'\times Y'$ along a twistor path on a moduli space of Hodge isometric pairs of  manifolds (see Section \ref{EPW2})  to a twisted sheaf on $X\times Y$. 
Finally, we study the Brauer class of the latter twisted sheaf and use a general result on deformation of Fourier-Mukai kernels (cf. Theorem \ref{main markman}) to prove that it induces a derived (or twisted derived) equivalence between $X$ and $Y$.

 The case of higher Picard rank is discussed in Section \ref{EPW4}. Here, new technical difficulties arise since we cannot compose a derived equivalence as in Theorem \ref{main1} with a twisted derived equivalence as in Theorem \ref{main2}, see \cite[Remark 1.3]{H}. In this case, we propose in Section \ref{convolution} another construction of derived equivalence along Lagrangian fibrations.
 This is a generalisation in Theorem \ref{Main} of \cite[Theorem A]{ADM} to the case of moduli spaces of twisted sheaves. 
 The construction works also in higher dimension and gives rise to various examples of twisted derived equivalent $K3^{[n]}$-type manifolds.
In particular, we are able to show the derived equivalence of the following known pair of \HK{} fourfolds with Picard rank $2$ and isomorphic transcendental lattices.
\begin{prop}
 A very general double EPW quartic, i.e.~a $K3^{[2]}$-type fourfold with Picard lattice $U(2)$, is derived equivalent to a special double EPW sextic, i.e.~a $K3^{[2]}$-type fourfold with Picard lattice $(2)\oplus (-2)$ such that the generator $(-2)$ is of divisibility $1$.
\end{prop}
The proof can be found in Section \ref{EPWquartic}. The technical difficulty is the computation in Proposition \ref{uni} of the Brauer class that is the obstruction on the double EPW sextic (isomorphic to the moduli space of twisted sheaves $M_{(2,2B+h,1)}(S,B)$) for the existence of the universal family on $S\times M_{(2,2B+h,1)}(S,B)$, where $(S,B)$ is the related twisted K3 surface of degree $2$.

Finally, we describe in Section \ref{convolution} a third construction of twisted derived equivalences that is a mixture of the above constructions.
We consider deformations of convolutions of two fiberwise equivalences along Lagrangian fibrations.
We infer in this way locally free Fourier-Mukai kernels that again deform as hyperholomorphic sheaves giving twisted equivalences of \HK\ manifolds of Picard rank $1$. Recall that from \cite{KP} the related dual Gushel-Mukai fourfolds (i.e.~associated to dual Lagrangian spaces) are derived equivalent. Here we deduce, both from Theorem \ref{main1} and the latter construction, that the two corresponding dual double EPW sextics are also derived equivalent. We end the paper with a list of questions and open problems.

\subsection*{Acknowledgement} We thank N.~Addington, E.~Markman, D.~Mattei, M.~Mauri, T.~Wawak and Ruxuan Zhang for helpful discussions and K.~Grzelakowski and C.~Tschanz for comments. \\
G.K.~is supported by the project Narodowe Centrum Nauki 2018/30/E/ST1/00530. M.K.~is supported by the project Narodowe Centrum Nauki 2018/31/B/ST1/02857.

\section{Preliminary}

In this section we introduce some technical tools that are needed in the paper.
The main method considered in this paper is the deformation a derived equivalence between two fixed \HK{} manifolds to a general pair. The problem is that we infer in this way twisted derived equivalences and we have to deal with derived categories of twisted sheaves (see \cite{Ca,Y}). 
For this reason we start by discussing Brauer classes on hyper-K\"ahler manifolds adapting the approach described in \cite{HS,vGK, K} from the case of K3 surfaces. 

 \subsection{Brauer classes}\label{brauerz} The cohomological Brauer group of a scheme $X$ is the group $$Br'(X)=H^2_{et}(\oo_X^{\ast})_{\operatorname{tors}}.$$ 
 It is known that for projective schemes this group is isomorphic to the usual Brauer group $Br(X)$ of equivalences of Azumaya algebras \cite{CE,GS}.
Let $X$ be a hyper-K\"ahler manifold with $H^3(X,\ZZ)=0$ (for example of $K3^{[n]}$-type) and $\beta\in  H^2(X,\oo_X^{\ast})_{\operatorname{tors}}$ be a Brauer class.
 From the exponential sequence

$$0  \to \mathbb Z   \to \mathcal{O}_X \xrightarrow{\exp(2\pi i(-))}  \mathcal{O}_X^{\ast}   \to 0, $$
and by the fact $H^3(X,\mathbb Z)=0$, we infer that there exists $\alpha\in H^2(X,\mathcal{O}_X),$ such that $$\exp(2\pi i(\alpha))=\beta.$$
If the order of $\beta$ is $k$ 
then $k\alpha$ is the image of some class $B_0\in H^2(X,\mathbb{Z})$ and we call $B=\frac{B_0}{k}\in \frac{1}{k}H^2(X,\mathbb{Z})$ a $B$-lift of $\beta$.
Moreover, we denote by $ [ \frac{B_0}{k} ] \in H^2(X,\oo_X^{\ast})$ the Brauer class corresponding to $\frac{B_0}{k}$.
Note that the above class $B$ is not unique and we have the ambiguity of adding a class in $H^2(X,\mathbb{Z})$ and a class in $H^{1,1}(X,\mathbb{Z}/k)$.  
The following is straightforward.
\begin{lem}
    The  Brauer class with $B$-lift $\frac{\delta}{2}$ such that $\delta$ is a $-2$ class of divisibility $2$ is uniquely defined, i.e.~does not depend of the choice of $\delta$.
\end{lem} 
\begin{proof}
 Observe that $H^2(X,\mathbb Z)$ decomposes as an orthogonal sum of a unimodular lattice and a 1-dimensional lattice generated by a -2 class of divisibility 2. It follows that any two -2 classes $\delta_1$, $\delta_2$ of divisibility 2 on $H^2(X,\mathbb Z)$ differ by a class from $2H^2(X,\mathbb Z)$ and hence $[\frac{\delta_1}{2}]=[\frac{\delta_2}{2}]$.
\end{proof}

\subsection{Deformation of a derived equivalence}\label{EPW2}

Nick Addington communicated to us the following general result concerning deformations of Fourier-Mukai kernels.
\begin{prop}\label{Nick}
For a flat family of coherent sheaves on a projective family of products of manifolds $X_t\times Y_t$ with $t\in B$ to be an equivalence (or
fully faithful) is an open condition.
\end{prop}
\begin{proof}
Let $X$ and $Y$ be smooth projective varieties.  Let $E$ in $D(X \times Y)$ induce a
functor $F: X \to Y$.  Then the left and right adjoints $L$ and $R$ are induced
by $E^{\ast}$ tensor $\omega_Y[\dim Y]$ and $\omega_X[\dim X]$, respectively, and the
unit and counit

       $$ \operatorname{id}_X \implies R \circ F$$
        $$ F \circ L \implies \operatorname{id}_Y$$

are induced by maps of kernels (\cite[Section 3.1]{CW}).  
Now $F$ is an equivalence if and only if this unit and counit are
isomorphisms, which happens if and only if the cones on the corresponding maps of kernels are zero.

In the case of smooth projective families $X \to B$ and $Y \to B$, we can act analogously, and we are asking about the locus in $B$
over which those cones are zero, i.e.~the complement of the image of their
support.  But the support is closed and our families are proper so the
complement of the image of the support is open.
\end{proof}

Note that the above proposition works only for local deformations and generally for coherent sheaves.
In this paper we need global deformations and work with twisted sheaves thus we consider below a finer variant of the previous proposition.

Our method consists of deforming a given Fourier-Mukai kernel giving a derived equivalence between two fixed hyper-K\"ahler manifolds that are Hilbert schemes of points on a K3 surface. Then spread it to twisted projectively hyperholomorphic sheaves over elements of the considered below moduli space of pairs of hyper-K\"ahler manifolds. 
The last step is to refine Proposition \ref{Nick} in this context and prove that it gives a twisted derived equivalence at every element in this moduli.
Let us be more precise; the following is build on \cite{Mar1} and we start by recalling some notions introduced there.

Denote by $\Lambda_n$ the Beauville Bogomolov Fujiki lattice of \HK{} fourfolds of $K3^{[n]}$-type. For  $\psi:\Lambda_n\otimes \QQ\to \Lambda_n\otimes \QQ$ a rational isometry,
we consider the moduli space $\M_{\psi}$ of pairs $X_t$ and $X_t'$ of marked \HK{} fourfolds of $K3^{[n]}$-type such that  $\psi$ represents via the markings a rational Hodge isometry between  $X_t$ and $X_t'$ (see \cite[Section 5.2]{Mar1} and \cite{O2}).
We assume that we have a twisted vector bundle $E_0$ on $X_0\times X_0'$ that gives a twisted derived equivalence and is compatible with this isometry \cite[Definition 5.20]{Mar1}.

Recall that a poly-stable bundle is a direct sum of stable bundles with the same slope.
It is proven that two points of $\M_{\psi}^0$ are connected by a sequence of generic twistor lines (\cite[Lemma 5.15]{Mar1}).
In this way, if $\E nd (E_0)$ is poly-stable hyperholomorphic with respect to an open set of K\"ahler classes as in \cite[Proposition 5.19]{Mar1} then for every $t$ in the connected component $\M_{\psi}^0$ of $\M_{\psi}$  there is a twisted vector bundle $E_t$ on $X_t\times X'_t$ such that the  Azumaya algebras $\mathcal{E}nd(E_t)$ and $\mathcal{E}nd(E_{0})$ are deformation equivalent.
By acting step by step we consider the case of a deformation along one twistor line as in \cite[Proposition 5.12]{MarB}). Note also that on one twistor line we obtain a twisted sheaf $\E$  such that $\mathcal{E} nd(\E)$ is a hyperholomorphic sheaf on this twistor family that restricts to $\E nd(E_0)$ above $X_0\times X_0'$.

 Recall also that a BKR equivalence is a vector bundle on $S^{[n]}\times M^{[n]}$ naturally obtained from a Fourier-Mukai kernel between two derived equivalent projective K3 surfaces $M$ and $S$
as in \cite[\S 7]{Mar1} (see also Section \ref{section r cyclic}).
Suppose $E_0$ on $X_0\times X_0'$ is obtained as a deformation along a twistor path of a BKR equivalence.
In this case, the sheaf $\E nd (E_0)$ is appropriately poly-stable hyperholomorphic from \cite[\S 6 and \S 8]{Mar1} thus for every $t\in \M^0_{\psi}$ gives rise to twisted bundles $E_t$ on $X_t\times X_t'$. 
For an introduction to twisted sheaves and derived categories of twisted sheaves see \cite{Ca,CS,HS,Y}.

\begin{thm}\label{main markman}
 Suppose that the sheaf $E_0$ on $ X_0\times X'_0$ is the Fourier-Mukai kernel of a BKR equivalence $D(X_0)=D(X_0')$. 
 Then, for $t\in \M_{\psi}^0$, a twisted bundle $E_t$ obtained via the above construction gives rise to a twisted derived equivalence between $D(X_t,\beta)=D(X_t',\beta')$ with twists $\beta\in Br(X_t)$, $\beta'\in Br(X_t')$ given by the Azumaya algebra $\E nd(E_t)$. 
\end{thm}
 \begin{proof}
 For $p\in X_0'$, denote by $E_0^p$ the fiber of $E_0$ above $X_0$. It follows from \cite[Theorem 1.4]{Mar2} that each $E_0^p$ is slope-stable for each K\"ahler class of an open subcone of the K\"ahler cone of $X_0$. 
Let $E_t^{p}$ be the twisted sheaf on $X_t$ that is the fiber of $E_t$ above $p$ for $p\in X_t'$. 
Note first that for, $p,q \in X_t'$, the sheaves $$\E nd(E^{p}_t)\ \text{and}\ \E nd(E^{q}_t)$$ are Morita equivalent so define a 
Brauer class $\beta \in Br(X_t)$ (the summand along $X_t$ of the Brauer class given by the Azumaya algebra $\E nd(E_t)$).

From \cite[3.2.1]{Ca1}, in order to show derived equivalence it is enough to show that $$Ext^k(E^p_{t},E^q_{t}) =0\ \text{for}\ k>0 \ \text{or}\ p\neq q\ \text{and}\ Hom(E^p_{t},E^p_{t})=1.$$
Arguing as in the proof of \cite[Theorem 1.4]{Mar2} (cf.~\cite[Theorem 1.6]{Mar2}) using \cite[Lemma 8.3]{Mar1}, we can find a special deformation $X_{t_0}$ of $X_0$  along a twistor path, such that $X_{t_0}$ is isomorphic to the square of a non-algebraic K3 surface and such that $E_{t_0}^p$ is poly-stable with respect to every K\"ahler class on $X_{t_0}$. 
 It follows from \cite[Corollary 8.1]{V1} that the dimensions of cohomologies of poly-stable hyperholomorphic vector bundles are invariant along twistor lines.
We infer {using the fact that $E_t^{p}$ is locally free so $\mathcal{E}xt^1(E_t^{p},E_t^{p})=0$ } that $$H^k(\E nd(E_t^{p}))= Ext^k(E^{p}_t,E^{p}_t)=H^k(\E nd(E_{0}^p))$$ for every $t\in X_0'$.

To consider different $p,q\in X'_t$, we study the cohomologies of the bundle $\mathcal{H}om(E^p_t,E^q_t)$.
As before $E^p_t$ and $E^q_t$ are projectively hyperholomorphic, meaning 
$\E nd(E^p_t)$ and $\E nd(E^q_t)$ are hyperholomorphic.
Thus they give rise to twisted vector bundles $E^p$ and $E^q$ on the twistor family.
Then $E^p\oplus E^q$ is also a twisted vector bundle on the twistor family therefore $\mathcal{E}nd(E^p\oplus E^q)$ is a vector bundle on the twistor family (so a hyperholomorphic bundle). 

In order to see that $\mathcal{E}nd(E^p_0\oplus E^q_0)$ is poly-stable with respect to K\"ahler classes in an open subcone of the K\"ahler cone of $X_0$ it is enough to observe that $E_0^p$ and $E_0^q$ are poly-stable with the same slopes. Thus, by \cite[Lemma 9.1]{Mar3} 
we have that $$\mathcal{H}om(E^p_0,E^q_0),\ \mathcal{H}om(E^q_0,E^p_0),\  \E nd (E_0^p )\ \text{and}\ \E nd (E_0^q )$$ split as direct sums of indecomposable slope stable bundles of the same slope.
It follows from  \cite[Corollary 8.1]{V1} that 
 the cohomologies $H^k(\mathcal{E}nd(E^p_t\oplus E^q_t))$ are invariant for all deformations of $X_0$ (see again the proof of \cite[Theorem 1.4]{Mar2}). We find in this way the cohomologies of the factors of
 $$\mathcal{E}nd(E^p_t\oplus E_t^q)=\mathcal{E}nd(E^p_t)\oplus \mathcal{E}nd( E_t^q)\oplus \mathcal{H}om (E^p_t, E^q_t)\oplus \mathcal{H}om (E^q_t,E^p_t)$$ 
 and conclude that $H^k(\mathcal{H}om (E^q_t,E^p_t))=0$ (since we know that $\mathcal{E}nd(E^p_t)$ and $\mathcal{E}nd(E^q_t)$ are hyperholomorphic and their cohomologies are invariant too).
 Finally, it is well known that the Brauer class of the twisted bundle $E_t$ is the Morita equivalence class of $\E nd(E_t)$ (see \cite[Theorem 1.3.5]{Ca1}).
\end{proof}
\begin{rem} We can naturally generalise the previous theorem for more general manifolds if we know that $E_0^p$ is poly-stable with respect to all K\"ahler classes.  
\end{rem}

\subsection{The Brauer class after deformation}  In this section we describe the Brauer twists $\beta\in H^2(\oo_{X_t}^{\ast})_{\operatorname{tors}}$ and $\beta'\in H^2(\oo_{X_t}^{\ast})_{\operatorname{tors}}$ of the deformed Fourier-Mukai kernel described in Theorem \ref{main markman} and study their $B$-lifts. 

Following \cite[\S 4]{Ca}, let $f:X\to S$ be a proper
smooth morphism of analytic spaces, and with $0$ a closed point of
$S$.  
Let $X_0$ be the fiber of $f$ over $0$.  We consider an element
$\alpha\in Br(X)$, such that $\alpha|_{X_0}$ is trivial as an element
of $Br(X_0)$, and we assume we are given a locally free
$\alpha$-twisted sheaf $\mathcal E$ on $X$.  Since $\alpha|_{X_0} = 0$, we
can modify the transition functions of $\E|_{X_0}$ in
such a way that we get an untwisted locally free sheaf $\E_0$ on
$X_0$.  

\begin{thm}{\cite[Theorem  4.1]{Ca}}\label{Cardelaru}
\label{deform}
Let $\E$ be a locally free $\alpha$-twisted sheaf on $X$, and let
$\E_0 = \E|_{X_0}$.  Assume that $S$ is small enough, so that we have an identification $H^i(X, \ZZ) \simeq
H^i(X_t, \ZZ)$ for all $i$ and all $t\in S$.  We assume that
$\alpha|_{X_0} = 0$, and therefore we can modify the transition
functions of $\E_0$ so that it is a usual sheaf on $X_0$.  Then we
have
\[ \alpha = \left [ -\frac{1}{\rk(\E_0)}c_1(\E_0) \right ]. \]

It follows that $\alpha|_{X_t}$ is the Brauer class of the general fiber. We find the corresponding $B$-lifts through the isomorphism $H^i(X, \ZZ) \simeq
H^i(X_t, \ZZ)$.
\end{thm}

Let us explain how we shall use this result in our situation. Recall that any pair of points in $\M_{\psi}$ are connected by a sequence of generic twistor lines.
In order to understand the Brauer class in an arbitrary deformation of $E_0$ it is enough to apply the above Theorem \ref{Cardelaru} along a generic twistor line. Indeed, it follows from \cite[6.7 (3)]{MarB} that the Azumaya algebra spreads in a canonical way along one twistor line. Above each point of this twistor line, this defines uniquely the Brauer class as a Morita equivalence of Azumaya algebras.
Now the $B$-lift of this Brauer class at a generic point of the twistor line is computed explicitly using Theorem \ref{Cardelaru} and is unique up to the ambiguity described in Section \ref{brauerz} so unique in $$(\frac{1}{\rk(E_0)} H^2(X_t,\ZZ ))/H^2(X_t,\ZZ).$$
 We conclude that the class in $(\frac{1}{\rk(E_0)} H^2(X_t,\ZZ))/H^2(X_t,\ZZ)$ of this $B$-lift is also unique at any point $t\in \M_{\psi}$ and is not dependent on the choice of the twistor path.

\section{Twisted derived equivalences of r-cyclic type}\label{section r cyclic} In this section, let us compare the (twisted) derived categories of  pairs of two \HK{} fourfolds of $K3^{[2]}$-type $X$ and $Y$ that have Hodge isometric transcendental lattices and the Hodge isometry relating them is represented by a single reflection $\rho_r$ with respect to a class of degree $-2r$. We will introduce and investigate special twisted Fourier-Mukai kernels relating their (twisted) derived categories that will be called of $r$-cyclic type and will be used as building blocks in our constructions of (twisted) derived equivalences. 

More precisely,  let us recall from \cite{Mar2} the following result.  Let $X$ and $Y$ be two hyper-K\"ahler manifolds  of $K3^{[n]}$-type.
 For $$f\colon H^2(X,\QQ)\to H^2(Y,\QQ)$$ a rational Hodge isometry, we say is of $r$-cyclic type, for some integer $r$, if there exists a primitive class 
 $u\in H^2(X,\ZZ)$ that satisfies $(u,u)=2r$ and $(u,  .)\in H^2(X,\ZZ)^{\vee}$
 primitive as well and there exists 
 $$g\colon H^2(X,\ZZ)\to H^2(Y,\ZZ),$$
 a parallel transport operator, such that $f=-g\rho_u$, where $\rho_u\in O(H^2(X,\QQ))$ is the reflection 
 $\rho_u(x)=x-\frac{2(u,x)}{(u,u)}u$.

\begin{thm}[Corollary 8.5 \cite{Mar1}] \label{Markman def} If $f\colon H^2(X,\QQ)\to H^2(Y,\QQ)$ is a rational Hodge isometry of $r$-cyclic type for two \HK\ manifolds $X$, $Y$ of $K3^{[n]}$-type for which there exists a K\"ahler class $\kappa \in H^2(X,\QQ)$ such that $f(\kappa)$ is a K\"ahler class, then there exists a locally free possibly twisted sheaf on $X\times Y$ of rank $n!r^{n}$ that is a deformation along a twistor path of a BKR Fourier-Mukai kernel between two manifolds $S^{[n]}$ and $M^{[n]}$ for two derived equivalent K3 surfaces $S$ and $M$. 
\end{thm}
Using Theorem \ref{main markman} a deformation along a twistor path of a Fourier-Mukai kernel of BKR type is a twisted Fourier-Mukai kernel. We will call such Fourier-Mukai kernels as well as related twisted derived equivalences of $r$-cyclic type.  We hence infer the following.
\begin{cor}\label{cor-B} For $X$ and $Y$ hyper-K\"ahler manifolds of $K3^{[n]}$-type  of Picard number one related by a rational Hodge isometry of r-cyclic type there exists a twisted derived equivalence of r-cyclic type: $D(X,\beta_X)\simeq D(Y,\beta_Y)$ for some Brauer classes $\beta_X \in H^2(\oo^{\ast}_X)$ and $\beta_Y \in H^2(\oo^{\ast}_Y)$.
\end{cor}

In our investigations we will deal with specific Fourier-Mukai kernels of $r$-cyclic type for which we will be able to compute explicitly the Brauer classes on both sides of the obtained equivalence. For that we need a more detailed description of BKR equivalences and their deformations along twistor paths.

\subsection{Deformation of BKR equivalences along twistor paths}\label{deforming BKR}
Let us fix two coprime numbers $r,s$ and a primitive hyperbolic sublattice $U_{e,f}$ of the K3 lattice $\Lambda_{K3}$ with generators $e,f$. Consider a very general $K3$ surface $S$ of degree $2rs$ polarised by a class $L$ with marking 
 $$\phi_S\colon H^2(S,\mathbb Z)\to \Lambda_{K3}$$
 such that $\phi_S(L)=e+rsf$.
Let  $M=M_S(r,L,s)$ be the moduli space of stable sheaves on $S$ with Mukai vector $(r,L,s)$. In this case $M$ is also a $K3$ surface of Picard number 1. Moreover, $S$ and $M$ are derived equivalent and $M$ admits a marking 
$$\phi_M\colon H^2(M,\mathbb Z)\to \Lambda_{K3}$$
 such that the reflection $\tilde \rho_r:\Lambda_{K3}\to \Lambda_{K3}$ with respect to a class $e-rf$ of degree $-2r$ via the markings restricts to a Hodge isometry between transcendental lattices $T_S$ and $T_M$.
    
   By \cite{P}, the \HK{} fourfolds $S^{[2]}$ and $M^{[2]}$ are also derived equivalent. Moreover, the derived equivalence is given by a so-called BKR functor whose Fourier-Mukai kernel $\mathcal K_{S^{[2]},M^{[2]}}$ is a vector bundle on $S^{[2]}\times M^{[2]}$ of rank $2r^2$ inducing a rational Hodge isometry of $r$-cyclic type, which composed with appropriate markings gives
   $$\rho_r\colon \Lambda\simeq  H^2(S^{[2]},\mathbb Q)\to H^2(M^{[2]},\mathbb Q)\simeq \Lambda,$$
   the reflection with respect to the class  $e-rf\in \Lambda := \Lambda_{K3}\oplus \langle -2\rangle$. 

Now, following \cite[Corollary 8.5]{Mar2} for every pair of marked \HK{} fourfolds $X$ and $Y$ for which $\rho_r$ induces a rational Hodge isometry we can find a deformation along a twistor path of $\mathcal K_{S^{[2]},M^{[2]}}$ to a twisted bundle $\mathcal K_{X,Y}$ on $X\times Y$. The latter, by Theorem \ref{main markman} induces a twisted derived equivalence with Brauer class given by the Azumaya algebra $\E nd \mathcal( \mathcal K_{X,Y})$ which is a deformation of the Azumaya algebra $\E nd( \mathcal K_{S^{[2]},M^{[2]}})$. We will use Theorem \ref{Cardelaru} to determine these twists in explicit cases that we shall consider below. 

We will have two types of twisted derived equivalences depending on the divisibility of the polarizations. We shall describe those equivalences in the next two subsections.

\subsection{ Case of divisibility 1}

Let us first consider the case of divisibility 1. For that recall that we fixed an embedding $\Lambda_{K3}\subset \Lambda$, two generators $e,f\in \Lambda $ of a primitive hyperbolic sublattice $U_{e,f}\subset \Lambda_{K3}\subset \Lambda$  and $\delta$ a class of square -2 and divisibility 2 orthogonal to $\Lambda_{K3}$. 

Moreover, for $r\in \mathbb Z$, we have 
$\rho_r: \Lambda\otimes \mathbb Q \to \Lambda\otimes \mathbb Q $ the reflection with respect to $e-rf$. In particular, 
$$\rho_r(e)=rf, \ \rho_r(f)=\frac{1}{r}e \text{ and } \rho_r(x)=x \text{ for } x\in U_{e,f}^\perp.$$ For $r,s,l\in \mathbb Z$ define     
$$L_{r,s,l}= e+rsf-l\delta,\ T_{r,s,l}=\langle L_{r,s,l}\rangle ^{\perp}\subset \Lambda,$$  
$$L'_{r,s,l}=se+rf-l\delta,\ T'_{r,s,l}=\langle L'_{r,s,l}\rangle ^{\perp}\subset \Lambda.$$ 
Note that $(L_{r,s,l})^2=(L'_{r,s,l})^2=2rs-2l^2$.

  Let $X$, $Y$ be two \HK{} fourfolds of $K3^{[2]}$-type with markings:
$$\phi_X\colon H^2(X,\mathbb Z)\to \Lambda$$
$$\phi_Y\colon H^2(Y,\mathbb Z)\to \Lambda$$

such that $$\phi_X(\operatorname{NS}(X))= \langle L_{r,s,l}\rangle ,\ \phi_Y(\operatorname{NS}(Y))=\langle L'_{r,s,l}\rangle$$ for some $r,s\in \mathbb Z$ satisfying $\gcd(r,s)=1$, and $l$ such that $r|2l$ and such that $\rho_r$ composed with the markings defines a rational Hodge isometry between $X$ and $Y$. Following Subsection \ref{deforming BKR} we can then construct a bundle $\mathcal K_{r,s,l}:=\mathcal K_{X,Y}$ on $X\times Y$ that defines a twisted derived equivalence $$D^b(X,\beta)=D^b(Y,\beta')$$ for some Brauer classes $\beta$, $\beta'$. Based on the fact that $\mathcal K_{r,s,l}$ is a deformation of  $\mathcal K_{S^{[2]},M^{[2]}}$, the following lemma provides us the descriptions of $\beta$ and $\beta'$ depending on $r$, $s$, $l$.

\begin{lem}\label{proof-div1}
    Let $X$, $Y$ be two \HK{} fourfolds of $K3^{[2]}$-type marked in such a way that $\operatorname{NS}(X)= \langle L_{r,s,l}\rangle $, $\operatorname{NS}(Y)=\langle L'_{r,s,l}\rangle$ for some $r,s\in \mathbb Z$ satisfying $\gcd(r,s)=1$, and $l$ such that $r|2l$. Assume that $\rho_r$ via chosen markings induces a Hodge isometry between $T_X$ and $T_Y$ and let $\mathcal K_{X,Y}$ arise in the construction above. Then we have the following two possibilities:
    \begin{enumerate}
        \item  $\frac{2l}{r}$ is odd and $\mathcal K_{X,Y}$ is the Fourier-Mukai kernel of a derived equivalence $D^b(X)\simeq D^b(Y)$,
        \item $\frac{2l}{r}$ is even and $\mathcal K_{X,Y}$ is the Fourier-Mukai kernel of a twisted derived equivalence $D^b(X,[\frac{\delta_X}{2}])\simeq D^b(Y,[\frac{\delta_Y}{2}])$.
    \end{enumerate}
    
\end{lem}
\begin{proof} By 
    Theorem \ref{Cardelaru}, since the twisted sheaf $\mathcal K_{X,Y}$ is obtained as a deformation of $\mathcal K_{S^{[2]},M^{[2]}}$ its twist is given by the Brauer class corresponding in 
   $$(\Lambda \oplus \Lambda)\otimes \mathbb Q=H^2(S^{[2]},\mathbb Q)\oplus H^2(M^{[2]},\mathbb Q)=H^2(X,\mathbb Q)\oplus H^2(Y,\mathbb Q),$$
   identified via markings to $\frac{c_1(\mathcal K_{S^{[2]},M^{[2]}})}{2r^2}$. The latter, by \cite[Lemma 11.1]{Mar2} and \cite[Eq. (7.1)-(7.2)]{Mar1} gives
   \begin{equation}\label{chern Lrsl}
      \frac{c_1(\mathcal K_{S^{[2]},M^{[2]}})}{2r^2}= (\frac{e+rsf}{r}-\frac{\delta}{2}, \frac{se+rf}{r}+\frac{\delta}{2})=(\frac{L_{r,s,l}}{r}+(\frac{l}{r}-\frac{1}{2})\delta, \frac{L'_{r,s,l}}{r}+(\frac{l}{r}+\frac{1}{2})\delta).
   \end{equation}

   From Theorem \ref{Cardelaru} and (\ref{chern Lrsl}) we infer that $\beta$ and $\beta'$ are Brauer classes admitting lifts to 
   $(\frac{l}{r}-\frac{1}{2})\delta_X$ and $(\frac{l}{r}+\frac{1}{2})\delta_Y$ respectively, where $\delta_X=\phi_X^{-1}(\delta)$ and $\delta_Y=\phi_Y^{-1}(\delta)$. In particular $\beta$ and $\beta'$ are trivial when $\frac{2l}{r}$ is odd and equal to  $[\frac{\delta_X}{2}]$ and $ [\frac{\delta_Y}{2}]$ respectively when $\frac{2l}{r}$ is even. This concludes the proof.
\end{proof}

\subsection{Case of divisbility 2}
Let us now construct r-cyclic type derived equivalences for $X$ and $Y$ of Picard number 1 and divisibility 2.  To do it let us define for $r,s\in \mathbb Z$ and $l$ odd:

$$\bar L_{r,s,l}= 2e+2rsf+l\delta,\ \bar T_{r,s,l}=\langle \bar L_{r,s,l}\rangle ^{\perp}\subset \Lambda,$$  
$$\bar L'_{r,s,l}=2se+2rf+l\delta,\ \bar T'_{r,s,l}=\langle \bar L'_{r,s,l}\rangle ^{\perp}\subset \Lambda.$$ 
Then $(\bar L_{r,s,l})^2=(\bar L'_{r,s,l})^2=8rs-2l^2=:2d$.

  \begin{lem}\label{proof-lem-div2}
    Let $X$, $Y$ be two \HK{} fourfolds of $K3^{[2]}$-type marked in such a way that $\Pic_X= \langle \bar L_{r,s,l}\rangle $, $\Pic_Y= \langle \bar L'_{r,s,l}\rangle$ for some $r,s\in \mathbb Z$ satisfying $\gcd(r,s)=1$, and $r, l$ odd such that $r|l$.  Assume that $\rho_r$ via the chosen markings induces a Hodge isometry between $T_X$ and $T_Y$ and let $\mathcal K_{X,Y}$ arise in the construction above. Then $\mathcal K_{X,Y}$ is the Fourier-Mukai kernel of an equivalence $D^b(X)\simeq D^b(Y)$.

\end{lem}  
\begin{proof}
    The proof follows  the lines of the proof of Lemma \ref{proof-div1}. For the Brauer classes, we have
       \begin{equation}\label{chern Lrsl div2}
      \frac{c_1(\mathcal K_{S^{[2]},M^{[2]}})}{2r^2}= (\frac{e+rsf}{r}-\frac{\delta}{2}, \frac{se+rf}{r}+\frac{\delta}{2})=(\frac{\bar L_{r,s,l}}{2r}-\frac{l+r}{2r}\delta, \frac{\bar L'_{r,s,l}}{2r}-\frac{l-r}{2r}\delta).
   \end{equation}
   Since $l$ and $r$ are odd and $r|l$ we have $\frac{l+r}{2r}\in \mathbb Z$ which leads to trivial Brauer classes on both sides.
\end{proof}

\section{Isometries between transcendental lattices of corank 1} \label{section isometries of transcendental}
The main idea of the proof of Theorems \ref{main1} and \ref{main2} will be to proceed as follows, start with a Hodge isometry between transcendental lattices of two \HK{} fourfolds of $K3^{[2]}$-type and Picard number one, extend it to a rational Hodge isometry between their second cohomology and write the latter as a composition of markings, explicit isometries of r-cyclic type studied in Section \ref{section r cyclic}  and integral isometries of the lattice $\Lambda=\Lambda_{K3}\oplus \langle-2\rangle$. Since the restriction to a corank 1 lattice $T$ of a rational isometry of $\Lambda$ that is integral on $T$ is determined up to composition by restrictions of integral isometries of $\Lambda$  by its action on the discriminant group of $T$, we will look for the decomposition from this point of view.

\begin{subsection}{Isometries between lattices orthogonal to a class of divisibility one in $\Lambda$}
In this section we prove the following proposition describing any isometry between transcendental lattices orthogonal to a class of divisibility one in $\Lambda$ in terms of a composition of restrictions of explicit isometries of r-cyclic type studied in Section \ref{section r cyclic}.

\begin{prop}\label{proof-lem4.8}
  Let $L_X$, $L_Y$ be classes of divisibility 1 and degree $2d$ in $\Lambda$ and $T_X$, $T_Y$ their orthogonal complements in $\Lambda$. Then for every isometric isomorphism $\phi:T_X\to T_Y$ we have two possibilities:
  \begin{enumerate}
      \item $d$ gives remainders $2$, $3$, $4$, $6$ or $7$ when divided by $8$. In that case, $\phi$ is equivalent up to composition on both sides with restrictions of autoisometries of $\Lambda$ to $$\rho_{r}: T_{r,s,l}\to T'_{r,s,l}$$ for some $r,s\in \mathbb Z$ satisfying $\gcd(r,s)=1$ and $l\in \mathbb Z$ such that $\frac{2l}{r}\in 2\mathbb Z$.
 \item $d$ gives remainders $0$, $1$ or $5$ when divided by $8$.
      In that case, there exists a sequence of classes $L_i$ for $i\in\{0,1, 2\}$ with orthogonal complement $T_i$, with $$L_0=L_X,\ 
      L_2=L_Y$$ and isometric isomorphisms $\phi_i:T_{i}\to T_{{i+1}}$ for $i\in \{0,1\}$ such that 
  $\phi=\phi_1\circ \phi_0$ and such that for each $i\in \{0,1\}$ the isometry $\phi_i$ is equivalent up to composing on both sides with restrictions of autoisometries of $\Lambda$ to $$\rho_{r_i}: T_{r_i,s_i,l_i}\to T'_{r_i,s_i,l_i}$$ for some $r_i,s_i\in \mathbb Z$ satisfying $\gcd(r_i,s_i)=1$, and $l_i\in \mathbb Z$ odd such that $\frac{2l_i}{r_i}\in 2\mathbb Z+1$.
    \end{enumerate}
\end{prop}

The proof, given at the end of the subsection, is based on the fact that isometries between sublattice $T_X$ and $T_Y$ modulo restrictions of autoisometries of $\Lambda$ are determined uniquely by their associated map between discriminants. Let us be more precise.

Recall that if  $L$ is a primitive class of divisibility 1 and degree $L^2=2d$ in the $K3^{[2]}$ lattice $\Lambda$ and   $T_L=L^{\perp}\subset \Lambda$, then $\operatorname{disc} T_L\simeq \mathbb Z_{2d}\oplus \mathbb Z_2$ and we have a homomorphism $$\gamma: \mathbb Z_{2d}=\operatorname{disc} L \to \operatorname{disc} T_X$$ defined by the property (see \cite[\S 1.5]{N})
 $$\Lambda=\{(p,q)\in \langle L\rangle ^*\oplus T_L^* \subset \Lambda_{\mathbb Q} |\ \gamma([p])=-[q] \}.$$ 
 Recall also that $\operatorname{disc} T_L$ is equipped with an intersection form $$q: \operatorname{disc} T_L\times \operatorname{disc} T_L\to \mathbb Q/\mathbb Z.$$ 
 Taking all this into account we have a unique decomposition $$\operatorname{disc} T_L=\mathbb Z_{2d}  \oplus \mathbb Z_2 ,$$ with $(1,0)=\gamma(1)$ and $(0,1)=[\delta_L]$ such that  $q_L([\delta_L], \gamma(1))=0$. Note that $$q(\gamma(1),\gamma_1)=-\frac{1}{2d}\ \text{and}\ q([\delta_L], [\delta_1])=-\frac{1}{2}.$$

 Let us now fix two primitive classes $L_X, L_Y\in \Lambda$ of divisibility 1 and degree $L^2=2d$. Denote by $T_X$, $T_Y$ their respective orthogonal complements, $\gamma_X$, $\gamma_Y$ respective embeddings of suitable discriminant groups, by  $q_X$, $q_Y$ the corresponding intersection lattices, and by $$[\delta_X]\in \operatorname{disc} T_X ,\ [\delta_Y]\in \operatorname{disc} T_Y$$ the classes in the discriminants which are orthogonal to $\gamma(1)$.  
 
Let $\phi: T_X\to T_Y$ be an isometry and let $$\bar{\phi}:\operatorname{disc}(T_Y)\to \operatorname{disc}(T_X)$$ be its induced action on the discriminant groups. Such an isometry can be naturally extended to an isometry of $\Lambda \otimes \QQ$ by extending this map with the identity to $$T_X\oplus <L_X> \to T_Y
\oplus <L_Y>.$$ On the other hand any automorphism of $\Lambda \otimes \QQ$ is composed of reflections and Hodge isometries as in \cite[\S 2]{Mar1}. 
Our aim now is to describe explicitly the reflections involved in our cases.

As discussed in \cite[\S 1.1]{W},  
$\phi$ is determined by 
$$\bar{\phi}(\gamma_Y(1))\in \operatorname{disc} T_X =\mathbb Z_{2d}\gamma_X(1)\oplus \mathbb Z_2 [\delta_X]$$ up to composing on both sides with autoisometries of $\Lambda$ preserving $T_X$ and $T_Y$ respectively. Our aim is thus to classify appropriate isomorphisms between the discriminant groups.

\begin{lem} \label{action on discr} Let us keep the notation above for two primitive classes $L_X, L_Y\in \Lambda$ of divisibility 1 and degree $L^2=2d$ and $\phi: T_X\to T_Y$ an isometry of their corresponding orthogonal complements. Moreover, let $(a,b)\in \mathbb Z_{2d}\oplus \mathbb Z_2$ represent the class $\bar{\phi}(\gamma_Y(1))$ under the identification $\operatorname{disc}(T_X)=\mathbb Z_{2d}\oplus \mathbb Z_2$ associated to $\gamma_X$. Then one of the following holds:
\begin{enumerate}
    \item $b= 0$ and $a^2\equiv 1 \mod 4d$
    \item $d\equiv 1 \mod 4$, $b=1$ and $a^2\equiv 3d+1 \mod 4d$
    \item $8|d$, $b=1$ and $a^2\equiv 3d+1 \mod 4d$.
\end{enumerate}
\end{lem}
\begin{proof}
    It is enough to observe that since $\phi$ is an isometry then also $\bar{\phi}$ is an isometry with respect to the intersection forms $q_X$ and $q_Y$. Hence $$\frac{-a^2}{2d}-\frac{b^2}{2}= q_X((a,b),(a,b))=q_Y(\gamma_Y(1), \gamma_Y(1))=-\frac{1}{2d}$$ which implies that $a^2\equiv -bd+1 \mod 4d$. 
    Note now that if $d$  and $a^2\equiv 3d+1 \mod 4d$ for some $a\in \mathbb Z$ then  $3d+1\equiv a^2 \mod 4$ which means $$3d+1\equiv 0 \mod 4,\ \text{i.e.}\ d\equiv 1 \mod 4.$$
    It remains to prove that if there exists $a$ such that $a^2\equiv 3d+1 \mod 4d$ and $d$ is even then $d$ is divisible by $8$. Indeed, if $d$ is even then $a$ must be odd and thus $(a-1)(a+1)$ is divisible by 8, which means that $d$ is also divisible by 8.
\end{proof}

On the other hand the map $T_{r,s,l}\to T'_{r,s,l}$ induced by the reflection $\rho_r$ has the following action on the discriminants.

\begin{lem}\label{action of rho r,s,l on discr}
    Let $r,s,l \in \mathbb Z$ such that $\gcd(r,s)=1$ and $t:=\frac{2l}{r}\in \mathbb Z$. Then 
    $\rho_r|_{T_{r,s,l}}$ defines an integral isometry $T_{r,s,l}\to T'_{r,s,l}$
    such that on discriminants we have
    $$\bar \rho_r(\gamma'_{r,s,l}(1))= (-\frac{1}{2}rt^2m+2sm-1) \gamma_{r,s,l}(1)+ \frac{2l}{r} [\delta_{r,s,l}]$$
    for any $m$ for which there exists $n\in \mathbb Z$ such that $rn+sm=1$. 
\end{lem}

\begin{proof}
Observe first that $T_{r,s,l}$ is generated by $C_1=e-rsf$, $C_2=2l f+ \delta$ and $U_{e,f}^\perp$. 
        Furthermore,  we have the following intersection matrix on $\langle C_1, C_2 \rangle$: 
        $$\left(\begin{array}{cc}
            -2rs &  2l\\
             2l & -2
        \end{array}\right).$$
        
        The canonical decomposition $\operatorname{disc} T_{r,s,l}=\mathbb Z_{2d}\oplus \mathbb Z_2 $ related to its embedding in $\Lambda$ is then given by  $$B_1=\frac{1}{2d}(C_1+lC_2),\ B_2=\frac{C_2}{2}$$ representing generators of both components. Indeed, both $B_1$ and $B_2$ represent elements in $(T_{r,s,l})^{*}$. Furthermore,  $[B_1]$ is of order $2d$ whereas $B_2$ is of order 2 and the product of $B_1$ and $B_2$ is $0$, hence $q_X([B_1],[B_2])=0$. Finally, we have  $\Lambda\ni f= \frac{L_{r,s,l}}{2d}-B_1$.
        
        Looking at the image of $T_{r,s,l}$ by $\rho_r$ we have  
        $$\rho_{r} (H+l\delta)=se+rf+l\delta \text{ and } \rho_r(T_{r,s,l})=\langle -se+rf, \frac{2l}{r} e+\delta \rangle + U_{e,f}^{\perp}.$$ In particular, $\rho_r(T_{r,s,l})=T'_{r,s,l}$. 
        
        Let us use the notation, 
        $$t:=\frac{2l}{r},\ C'_1:=\rho_r(C_1)=-se+rf \text{ and } C_2':=\rho_r (C_2)=t e+\delta .$$ Then the generator for the $\mathbb Z_{2d}$ part of the discriminant  group of $T'_{r,s,l}$ related to the embedding of $ T'_{r,s,l}\subset \Lambda$  is  
        $$\gamma'(1)=\frac{1}{2d}((rn-sm+\frac{1}{2}rt^2m)C_1'+\frac{1}{2}rtC_2'),$$ where $n,m\in \mathbb Z$ are any numbers such that $rn+sm=1$. 
        Now $$\bar\rho_r (\gamma'(1))=\frac{1}{2d}((rn-sm+\frac{1}{2}rt^2m)C_1+\frac{1}{2}rtC_2).$$
        It remains to compute the decomposition of $\bar\rho_r (\gamma'(1))$ in terms of $B_1$, $B_2$.
        This can be done by computing the intersections with $B_1$ and $B_2$ using the intersection form $q_{r,s,l}$. 
    We get: $$q_{r,s,l}(\bar\rho_r (\gamma'(1)), B_2)=   q_{r,s,l} (\frac{1}{2d}((rn-sm+\frac{1}{2}rt^2m)C_1+\frac{1}{2}rtC_2),\frac{1}{2}C_2)=\frac{1}{2}tm\in \mathbb Q/\mathbb Z$$

    $$q_{r,s,l}(\bar\rho_r (\gamma'(1)), B_1)= q_{r,s,l} (\frac{1}{2d}((rn-sm+\frac{1}{2}rt^2m)C_1+\frac{1}{2}rtC_2),$$ 
    $$\frac{1}{2d}(C_1+lC_2))= \frac{\frac{1}{2}rt^2m-2sm+1}{2d}.$$

    We conclude that $\bar\rho_r (\gamma'(1))=\frac{-\frac{1}{2}rt^2m+2sm-1}{2d} B_1+\frac{1}{2}tm B_2$.
\end{proof}

We are now ready to prove the Proposition \ref{proof-lem4.8}.

\begin{proof}[Proof of Proposition \ref{proof-lem4.8}]
  Recall that up to composing with restrictions of isometries of $\Lambda$ the isometry $\phi$ is determined by its action on the discriminant, or even more precisely by the pair $$(a,b):=\bar{\phi}(\gamma_Y(1))\in \mathbb Z_{2d}\oplus \mathbb Z_2.$$ It is thus enough to prove that each type of action on the discriminant from Lemma \ref{action on discr} is realised as a composition of actions  $$\rho_{r}:T_{r,s,l}\to T'_{r,s,l}$$ as in the lemma.
   According to Lemma \ref{action on discr} we have three possibilities for $(a,b)$ that we consider separately.
Let us first consider the following case. 
\begin{description}
    \item[Case $b=0$]   In that case $a^2\equiv 1 \mod 4d$, and we have no restriction on $d$. 
 Observe that $a$ corresponds to a decomposition 
      $d=r k$ for some $r,k$ coprime such that there exists $m, n\in \mathbb Z$ with $rn+km=1$ and $a=km-rn=2km-1$. Then fix $l=r$ and $s=r+k$ so that $d=rs-l^2$. We now have $L_{r,s,l}^2=2d$, $r,s$ are coprime and $$r(n-m)+sm= rn+km=1.$$ Furthermore,  $l$ is odd and $$t=\frac{2l}{r}=2\in 2\mathbb Z.$$ By Lemma \ref{action of rho r,s,l on discr} we have, $$\bar{\rho_r}(\gamma'_{r,s,l})=(-\frac{1}{2}rt^2 m+2sm-1,tm)=(2(s-r)m-1,0)=(2km-1,0)=(a,0).$$
    
\end{description}

Before we pass to other cases let us first prove the following claim:
for a fixed $d$, if $a_1, a_2\in \mathbb Z$ satisfy $$a_1^2\equiv a_2^2\equiv 3d+1 \mod 4d$$ then there exists $a_3\in \mathbb Z$ such that $$a_3^2\equiv 1 \mod 4d$$ and $$a_3a_1\equiv a_2 \mod 4d.$$ Indeed, in that case $(a_1-a_2)(a_1+a_2)\equiv 0 \mod 4d$. It follows that both $a_1-a_2$ and $a_1+a_2$ are even and $$\frac{(a_1-a_2)}{2}\frac{(a_1+a_2)}{2}\equiv 0 \mod d.$$ Let $r=\gcd (\frac{a_1-a_2}{2}, d) $ and $s=\gcd (\frac{a_1+a_2}{2},d)$. Then $$\gcd(r,s)= \gcd(\frac{a_1-a_2}{2}, \frac{a_1+a_2}{2},d)=\gcd(\frac{a_1-a_2}{2}, a_2,d)=1$$ since $\gcd(a_2, d)=1$. It follows that a pair $a_1$,$a_2$ as in the assumptions induces a decomposition of $d$ as a product of coprime numbers. In consequence the number of solutions of the equation $$a^2\equiv 3d+1 \mod 4d$$ if nonzero is equal to the number  of solutions of $a^2\equiv 1 \mod 4d$. Finally, products of solutions of the two equations are solutions of the first equations. This proves the claim. 

As a consequence, for each of the remaining cases we need only to find a realisation for one $a$. The remaining automorphisms corresponding to other solutions of $$a^2\equiv 3d+1 \mod 4d$$ will then be obtained as compositions of the automorphisms associated to solutions of $a^2\equiv 1 \mod 4d$ with the one automorphism constructed in each case. Note that the existence of $a$ such that $a^2\equiv 3d+1 \mod 4d$ implies that $d$ is divisible by 8 or gives rest 1 when divided by 4 (see proof of Lemma \ref{action on discr}). 
      
\begin{description}
      \item[Case $b=1$ and $d\equiv 1 \mod 4$] In this case, $a^2\equiv 3d+1 \mod 4d$. We write $d=uv$, with $u$, $v$ coprime. Assume that $up+vq=1$ and set      
      $r=2u$, $l=u$ and $s=\frac{v+u}{2}$. Then $t=\frac{2l}{r}=1$  is odd  and $r(\frac{p+q}{2})-s q=1$, hence we can take $m=-q$ when applying Lemma \ref{action of rho r,s,l on discr} to
      $$\rho_{r,s,l}: T_{r,s,l}\to T'_{r,s,l},$$ We get 
      $\bar{\rho}_r(r,s,l)=(-vq-1,1)$. Observe now that for two different decompositions $d=u_iv_i$  with $u_i$,$v_i$ odd and $\gcd(u_i,v_i)=1$ for $i=1,2$ the number $-v_iq_i-1$ obtained as above are distinct. Indeed, for two distinct decompositions, up to exchanging their numeration one always finds $\alpha$ a prime number such that $\alpha$ divides $u_1$ and $v_2$ but not $u_2$ nor $v_1$. In that case $\alpha$ divides $v_2q_2$ but does not divide $v_1q_1=1-u_1p_1$, which implies that $-v_1q_1-1$ and $-v_2q_2-1$ are distinct modulo $2d$.
      
      In this way, from actions of $\rho_{r,s,l}$ on the discriminant  we recover as many pairs $(a,1)$ as there are solutions to the congruence $a^2\equiv 3d+1 \mod 4d$. It follows that for any $a$ we can find suitable $r,s,l$.
      \item[Case $b=1$ and $8|d$]  Here we also have $a^2\equiv 3d+1 \mod 4d$. In that case, take any decomposition $d=2^{w+3}uv$ with $u$,$v$ odd and coprime and choose $p,q\in \mathbb Z$ such that $up+vq=1$. We then take $$l=2^{w+2}u,\ r=2^{w+3}u,\ s=v+2^{w+1}u,\ t=1.$$ Then $r$ and $s$ are coprime hence there exists $m$, $n \in \mathbb Z$ such that $rn+sm=1$. We then apply Lemma \ref{action of rho r,s,l on discr} and get 
 $$\bar{\rho}_r(\gamma'_{r,s,l})=(2vm-1,1).$$ It remains to observe that for two different decompositions $d=2^{w+3}u_iv_i$  with $u_i$,$v_i$ odd and $\gcd(u_i,v_i)=1$ for $i=1,2$ the numbers $2v_im_i-1$ obtained as above are distinct. Indeed, for two distinct decompositions, up to exchanging their numeration one always has $\alpha$ a prime number such that $\alpha$ divides $u_1$ and $v_2$ but not $u_2$ nor $v_1$. In that case $\alpha$ divides $2v_2m_2$ but does not divide $2v_1m_1$ since:
 $$2v_1m_1=2(s_1-2^{w+1}u_1)m_1=2-2r_1n_1-2^{w+2}u_1m_1=2-u_1(2^{w+4}n_1+2^{w+2}m_1).$$
 It follows that  $2v_1m_1$ and $2v_2m_2$ are distinct modulo $2d$. Since the number of possible $a\in \mathbb Z_{2d}$ satisfying $a^2\equiv 3d+1 \mod 4d$ is equal to the number of decompositions of $d$ into a product of coprime numbers we can obtain each pair $(a,1)$ as corresponding to an action on the discriminant of $\rho_{r,s,l}$ for some $r,s,l$ satisfying the required conditions, i.e. $r,s$ coprime and $t=\frac{2r}{l}\in 2\mathbb Z+1 $. 
\end{description}

 Putting all  the above together:
 \begin{enumerate}
     \item When $d$ gives remainders  $2$, $3$, $4$, $6$ or $7$ when divided by 8 then by Lemma \ref{action on discr}, any automorphism $\phi:T_X\to T_Y$ corresponds to $(a,b)\in \mathbb Z_{2d}\oplus \mathbb Z_2$ such that $b=0$. In that case, $\phi$ is equivalent to a reflection $\rho_{r}: T_{r,s,l}\to T'_{r,s,l}$ for some $r,s\in \mathbb Z$ satisfying $\gcd(r,s)=1$ and $l\in \mathbb Z$ such that $\frac{2l}{r}\in 2\mathbb Z$, as in the assertion.

     \item When $d$ gives remainders  $0$, $1$, $5$ modulo 8, then there exists a solution $a$ for the congruence $a^2\equiv 3d+1 \mod 4d$. Furthermore, for every such solution we find $r,s,l$ such that $\frac{2l}{r}\in 2\mathbb Z+1 $, and the action of  $\rho_{r,s,l}$ on the corresponding discriminants is given by  $(a,1)$. It remains to observe that by composing two actions corresponding to  $(a_1,1)$ and $(a_2,1)$ we can get an action corresponding to $(a,0)$ for any $a$ such that $a^2\equiv 1 \mod 4d$.
      \end{enumerate}\end{proof}

\end{subsection}

\begin{subsection}{Isometries between transcenedental lattices orthogonal to a class of divisibility 2 in $\Lambda$} 
In this section we consider the case of isometry between two lattices $$T_X=L_X^{\perp},\ T_Y=L_Y^{\perp} \subset \Lambda,$$ where $L_X$ and $L_Y$ are classes of divisibility 2 in $\Lambda$.
We provide an analogous statement to Proposition \ref{proof-lem4.8}. In this case, the assertion is simpler, as we prove that every isometric isomorphism between transcendental lattice $\phi:T_X\to T_Y$ can in fact be realised as a restriction of a rational isometry of r-cyclic type.

\begin{prop}\label{proof-div2}
  Let $L_X$, $L_Y$ be classes of divisibility 2 and degree $2d$ in $\Lambda$ and $T_X$, $T_Y$ their orthogonal complements in $\Lambda$. Then every isometric isomorphism $\phi:T_X\to T_Y$ is equivalent up to composing on both sides with restrictions of autoisometries of $\Lambda$ to $\rho_{r}: \bar{T}_{r,s,l}\to \bar{T}'_{r,s,l}$ for some $r,s\in \mathbb Z$ satisfying $\gcd(r,s)=1$, and $l$ odd such that $r|l$.
\end{prop}
For the proof we will again work with discriminants. In this case, contrary to the case of divisibility 1, for a class $L$ of divisibility 2 we have $L^2=2d$ with $d\equiv 3 \mod 4$. Moreover,  $\operatorname{disc} \langle L\rangle \simeq \mathbb Z_{2d}\simeq \mathbb Z_2\oplus \mathbb Z_d$ and the discriminant of the transcendental lattice $\operatorname{disc} T_L=\mathbb Z_d$. In that case there is an embedding $\gamma: \operatorname{disc} T_L\to \operatorname{disc} \langle L\rangle $  such that 
 $$\Lambda=\{(p,q)\in \langle L\rangle ^*\oplus T_L^* \subset \Lambda_{\mathbb Q} |\ \gamma_L([q])=-[p] \}.$$ In fact $\gamma_L$ is an isometry between $\operatorname{disc} T_L$ and the subgroup of elements in $\operatorname{disc} \langle L\rangle $ whose order divides $d$.
 
 In this case, an isometry $\phi: T_L\to T_{L'}$ extends to an autoisometry of $\Lambda$ if and only if 
 $\gamma_L\circ \bar{\phi}=\gamma_L'$, where $\operatorname{disc} \langle L\rangle$ is identified with $\operatorname{disc} \langle L'\rangle$ by identifying $[\frac{L}{2d}]=[\frac{L'}{2d}]$. In consequence a map of lattices $\phi: T_L\to T_{L'}$ is determined up to composing on both sides with automorphisms of $\Lambda$ by $\gamma_L (\bar{\phi}( \gamma_{L'}^{-1}(\frac{L'}{d})))$. Note here that $\frac{L'}{d}$ as an element of order $d$ is in the image of $\gamma_{L'}$. 

The following gives a description of possible isometries between the discriminant groups that could arise as isometries induced by some $\phi: T_L\to T_{L'}$.

 \begin{lem}\label{aut div 2}
  If $L$, $L'$ are classes of divisibility 2 and degree $2d$ and $\phi: T_L\to T_{L'}$ is an isometry of the corresponding orthogonal complements. Then $$\gamma_L (\bar{\phi}( \gamma_{L'}^{-1}(\frac{L'}{d})))= a [\frac{L}{d}]$$ where 
  $$a^2\equiv 1 \mod d.$$
 \end{lem}
 \begin{proof} As we have observed above, the image of $\gamma_L$ is equal to the locus of elements of order dividing $d$ in  $\operatorname{disc}(\langle L\rangle)$, i.e. is generated by $[\frac{L}{d}]$.
     Recall now that $\gamma_L$ and $\gamma_{L'}$ are antiisometries with respect to the lattice structures on $\operatorname{disc}(\langle L\rangle)$ which implies that 
     $$a^2 \frac{2d}{d^2}= \frac{2d}{d^2}  \in \mathbb Q/2\mathbb Z.$$ It follows that $\frac{2a^2-2}{d}\in 2\mathbb Z$ which implies $a^2\equiv 1 \mod d$.
 \end{proof}

It remains to compare the isometries from Lemma \ref{aut div 2} with isometries arising from the restrictions of maps $\rho_r$ to $\bar T_{r,s,l}$. For that we need to understand the latter.

\begin{lem}\label{div 2 action on discr} For $r,s\in \mathbb Z$ coprime, with $r$ odd, $n, m\in \mathbb Z$ such that $rn+2sm=1$ and $l$ odd such that $r|l$, the map $\phi=\rho_r|_{\bar T'_{r,s,l}}$ is an integral isometry $\bar T'_{r,s,l}\to \bar T_{r,s,l}$ such that 
$$\gamma_{\bar{L}'_{r,s,l}} (\bar{\phi}( \gamma_{\bar{L}_{r,s,l}}^{-1}(\frac{\bar{L}_{r,s,l}}{d})))= (1-\frac{dm}{r})[\frac{\bar{L}'_{r,s,l}}{d}].$$

\end{lem}

\begin{proof}Let us denote 

$$B_1:=e-rsf \text{ and } B_2=lf+\delta.$$
Then
$$\bar T_{r,s,l}= \langle B_1, B_2 \rangle,$$
and it has $M$ as the intersection matrix
$$M=\left(\begin{array}{cc}
   -2rs  & l \\
  l   & -2
\end{array}\right),$$
with determinant $4rs-l^2= \bar{L}_{r,s,l}^2=2d$ 
and $$\operatorname{disc}(\bar T_{r,s,l})=\langle \frac{2B_1+lB_2}{d}\rangle.$$
Let now $C:=\frac{2B_1+lB_2}{d}=\frac{2e+(l^2-2rs)f+l\delta}{d}$, then 

$$\gamma_{\bar{L}_{r,s,l}}(C)= \frac{\bar{L}_{r,s,l}}{d}. $$
Indeed, $$C-\frac{\bar{L}_{r,s,l}}{d}=\frac{2e+(l^2-2rs)f+l\delta}{d}-\frac{2e+2rsf+l\delta}{d}=f\in \Lambda.$$

Similarly, if we denote $B'_1=\rho_r(B_1)=-se+rf$ and $B'_2=\rho_r(B_2)=\frac{l}{r}e+\delta $, then 
$$\bar T'_{r,s,l}= \langle B'_1, B'_2 \rangle,$$ with the same matrix $M$ as intersection matrix. Furthermore,  if we denote $C':=\bar{\phi}(C)$.
$$\operatorname{disc}(\bar T'_{r,s,l})=\langle \frac{2B'_1+lB'_2}{d}\rangle=\langle C'\rangle.$$

We now have
$$C'-(1-\frac{dm}{r})  \frac{\bar{L}'_{r,s,l}}{d}=\frac{(\frac{l^2}{r}-2s)e+2rf+l\delta }{d}-\left(1-\frac{dm}{r}\right)\frac{2se+2rf+l\delta}{d}=\frac{2sm-1}{r} e + 2mf+ m\delta.$$

It follows that $\gamma_{\bar{L}'_{r,s,l}}(C')=(1-\frac{dm}{r}) \frac{\bar{L}'_{r,s,l}}{d}$.   
\end{proof}

\begin{proof}[Proof of Proposition \ref{proof-div2}] Recall that isometries between transcendental lattices are determined up to autoisometries of $\Lambda$ by their action on the discriminant. Hence we just need to find $r,s,l$ such that $\rho_{r}: \bar{T}_{r,s,l}\to \bar{T}'_{r,s,l}$ acts on the discriminant groups in the same way as $\varphi$.
Note that $a$ such that $$a^2\equiv 1 \mod d$$ for odd $d$ induces a decomposition $d=rk$ such that $r$ and $k$ are coprime and both odd.   Indeed, $$\gcd(a-1,a+1)\leq 2$$ so every prime power dividing $d$ must divide either $a-1$ or $a+1$ thus if we set $$r=\gcd(a-1,d),\  
k=\gcd(a+1,d)$$ we get a decomposition $d=rk$ as described and thus when $$d\equiv 3 \mod 4$$ we have $r+k\equiv 0 \mod 4$. Let us now set $l=r$ and $s=\frac{k+r}{4}$. Then 
$d=4rs-l^2$, $\gcd(r,s)=1$, $l$ is odd  and $r|l$. In that case in Lemma \ref{div 2 action on discr}, the associated action on the discriminant is given by multiplication by $$1-\frac{dm}{r}=1-(4s-r)m=1-km $$ for $m$ such that  $2sm-1$ is divisible by $r$. It remains to observe in the same way as in the proof of Lemma \ref{action of rho r,s,l on discr} that different decompositions $d=r_ik_i$ lead to numbers $k_im_i$ which are different modulo $d$. Indeed, for two different decompositions (up to changing their order) we find a prime $\alpha$ which divides $r_1$ and $k_2$ but neither $r_2$ nor $k_1$. In that case $\alpha$ divides $k_2m_2$ but not $k_1m_1$. To see the latter we observe that $\alpha$ dividing $r_1$ divides $2s_1m_1-1$ and hence cannot divide $m_1$. Since it also does not divide $k_1$ and is prime we conclude that it does not divide $k_1m_1$. It follows that different decompositions $d=rk$ lead to different actions on the discriminant. Since the number of possible actions on the discriminant up to composing with autoisometries of $\Lambda$ is equal to the number of decompositions of $d$ into a product of coprime numbers we conclude the proof.
\end{proof}

\end{subsection}
    
\section{Derived equivalent $K3^{[2]}$-type manifolds of Picard rank $1$}\label{Sec4}
In this section we prove the main result of the paper, which includes Theorems \ref{main1} and \ref{main2} and is reformulated as follows.
\begin{thm}\label{MAIN}
    Let $X$ be a \HK\ fourfold of $K3^{[2]}$-type of Picard rank $1$ with $\operatorname{Pic}(X)=\langle L_X\rangle$. Let $Y$ another \HK\ fourfold of $K3^{[2]}$-type  such that the transcendental lattices  $T_X$ and $T_Y$ are Hodge isometric. Then  we have two possibilities:
 \begin{enumerate}

 \item  The class $L_X$ is of divisibility 1 and $d\equiv 1 \ \mod\ 4$ or $8| d$ and $X$ and $Y$ are derived equivalent, i.e.~$D^b(X)=D^b(Y)$. In this case,
         the number of Fourier-Mukai partners of $X$ is $2^{\tau (d)}$
        where $\tau(d)$ is the number of different prime divisors of $d$. 

\item The class $L_X$ is of divisibility 2 and $X$ and $Y$ are derived equivalent. The number of Fourier-Mukai partners of $X$ is then $2^{\tau (d)}$.
 \item The class $L_X$ is of any other type and the varieties $X$ and $Y$ are twisted derived equivalent with twist $\frac{\delta}{2}$,  i.e.~$D^b(X,[\frac{\delta_X}{2}])=D^b(Y,[\frac{\delta_Y}{2}])$.
   
        \end{enumerate}
\end{thm}
    
The proof will be given in Subsection \ref{proof}. Note that the theorem could be naturally generalised to higher dimension for $K3^{[n]}$-type manifolds and we hope to work on it in the future.
This proof is based on the following approach. For a pair of varieties $X$ and $Y$ as in the assumptions we will construct a Fourier-Mukai kernel defining their derived equivalence as an appropriate convolution obtained by means introduced in Section \ref{section isometries of transcendental} of special Fourier-Mukai kernels that are constructed and studied in Section \ref{section r cyclic}. For clarity of the argument let us first present the proof of derived equivalence of two dual double EPW sextics.

\subsection{A baby case: dual double EPW sextics} In order to illustrate the construction let us first describe derived equivalences between dual EPW sextics.
The Picard lattice of a $K3^{[2]}$-type \HK{} manifold is the following $$H^2(X_i,\mathbb{Z})\simeq U^3\oplus E_8(-1)^2\oplus (-2)=:\Lambda.$$ Denote by $e,f$ the generators of the first $U$ and $\delta$ that of the last $(-2)$.

Let $$\psi \colon \Lambda_{\QQ}\to \Lambda_{\QQ}$$ be a fixed isometry.
We consider as in \cite[\S 5]{Mar1} the moduli space 
$\mathcal{M}_{\psi}$ that
parametrises isomorphism classes of quadruples
($X_t,\eta_t,X_t',\eta_t')$ of deformation equivalent marked pairs $(X_t,\eta_t)$, where $\eta_t\colon H^2(X_t,\mathbb{Z}) \to \Lambda$ and $\eta_t'\colon H^2(X_t',\mathbb{Z}) \to \Lambda$ are
isometries, such that 
$$(\eta_t')^{-1}\psi \eta_t \colon H^2(X_t,\mathbb{Q})\to H^2(X_t',\mathbb{Q})$$
is a Hodge isometry which maps some K\"ahler class to a K\"ahler class.  
Markman proves that if we find a Fourier--Mukai kernel that is locally free and compatible in the sense of \cite[Definition 5.20]{Mar1} with $\psi$ then this bundle deforms to a twisted vector bundle for any pair in a given connected component $\M^0_{\psi}$ of $\M_{\psi}$ by \cite[Proposition 5.21]{Mar1}.
We study this moduli space knowing from \cite[Lemma 5.14]{Mar1} that the natural forgetful maps from $\M^0_{\psi}$ to the moduli space of marked \HK{} fourfolds of $K3^{[2]}$-type is surjective (for a given choice of orientation). 

As a warm up let us consider the simplest case. Let $$\psi:\Lambda_{\mathbb Q}\to \Lambda_{\mathbb Q}$$ be the isometry defined as follows. 
It interchanges $e-f$ and $\delta$ and is an identity elsewhere.
This isometry is in fact the reflection $\rho_{\delta-e+f}$ through the $-4$ class $\delta-e+f$.
As observed in \cite[Section 4.2]{O1} the isometry induces on the moduli space of marked $K3^{[2]}$-type manifolds  of degree $2$ (i.e.~double EPW sextics) with polarisation $e+f$ an involution corresponding to taking dual double EPW sextics. Observe furthermore, that from Lemma \cite[Lemma 2.4]{Mar1} there exists automorphisms $\alpha_1$, $\alpha_2$ of $\Lambda$ such that $\alpha_1\circ \rho_{\delta-e+f}\circ \alpha_2=\rho_{e+2f}$.

\begin{prop} 
Two dual double EPW sextics $X$ and $ X'$ are
derived equivalent, i.e. $D(X)\simeq D(X')$. 
\end{prop}
\begin{proof}
    By \cite[\S 5.5]{Mar1} any component of $\mathcal{M}_{\psi}$ induces pairs that give all possible
    Hodge isometries induced by $\psi$.
It is enough to construct one compatible locally free Fourier-Mukai kernel for a special pair of dual double EPW sextics.

Let $(S_{20},h)$ be a general degree $20$ K3 surface.
On $S_{20}^{[2]}$ we have a natural degree $2$ polarisation $\mu(h)-3\delta$.
The moduli space $M_{20}=M_{(2,h,5)}(S_{20})$ is also a degree $20$ K3 surface that is 
derived equivalent to $S_{20}$ but not isomorphic to it. In fact, there exists a locally free rank $2$ Fourier-Mukai kernel on $M_{20}\times S_{20}$.
From \cite{P} (see \cite[\S 9.2]{T}) this sheaf, induces a derived equivalence induced by a locally free rank $8$ vector bundle $\mathcal U$ on $M_{20}^{[2]}\times S_{20}^{[2]}$. 
By \cite[Corollary 7.3]{Mar1} this bundle, induces a Hodge isometry $$\psi \colon H^2(M_{20}^{[2]},\QQ)\to H^2( S_{20}^{[2]}, \QQ)$$ that is of 4-cyclic type, i.e.~there exist markings $$\eta_M: H^2(M_{20}, \mathbb Z)\to \Lambda,\ \text{and}\ \eta_S: H^2(S_{20}, \mathbb Z)\to \Lambda$$ such that  $\eta_S^{-1}\circ \rho_{e+2f}\circ \eta_M=\psi$.
Moreover, if we denote by $h_M$, $\delta_M$, $h_S$, $\delta_S$ the natural generators of $\Pic M_{20}^{[2]}$ and $\Pic S_{20}^{[2]}$ respectively then we can assume $$\eta_M(h_M)=e+10f, \ \eta_M(\delta_M)=\delta\ \text{and}\ \eta_S(h_S)=2e+5f,\ \eta_S(\delta_S)=\delta.$$

From Theorem \ref{Markman def}, by deforming $\U$ along twistor paths there exists a possibly twisted vector bundle on the product $X_t\times X_t'$ for each pair $(X_t,X_t')$ in $\M_{\psi}$. In order to conclude by Theorem \ref{main markman} we need to find the twist of a deformation when $X_t$ and $X_t'$ are dual EPW sextics.
From \cite[Lemma 11.1]{Mar2} we have in general $$\frac{c_1(\mathcal U)}8=\left(\frac{h_M}{2}-\frac{\delta_M}{2},\frac{h_S}{2}-\frac{\delta_S}{2}\right).$$

Note that a general twistor path after taking out one point in each twistor line is contractible. Then the cohomolgy groups $X_t\times X_t'$ along the contractible subset of the twistor path can be identified. From \cite[Theorem  4.1]{Ca} the topological Brauer class of the twist of the general deformation of the Fourier-Mukai kernel along that open subset of the twistor path is $$\left(\frac{h_M}{2}-\frac{\delta_M}{2},\frac{h_S}{2}-\frac{\delta_S}{2}\right).$$ Since every double EPW sextic $X$ polarised by $L$ can be marked by $\eta:H^2(X,\mathbb Z) \to \Lambda$ in such a way that $\eta(L)=e+10f-3\delta$, it  appears as $X_t$ in $M_{\psi}$ in such a way that $X'_t$ is the dual EPW sextic. Then the class $(\frac{h_M}{2}-\frac{\delta_M}{2},\frac{h_S}{2}-\frac{\delta_S}{2})$ will be in $$H^2(X_t \times X_t',\mathbb{Z})+\frac{1}{2}\Pic(X_t\times X'_t)$$ hence will represent the trivial Brauer class and define a derived equivalence.
Finally, by \cite[Proposition 1.2]{MMY}, we know that $M^{[2]}_{20}$ and $S_{20}^{[2]}$ are not birational (thus the isometry of the transcendental lattices is nontrivial).
\end{proof}

\subsection{Proof of Theorem \ref{MAIN}}\label{proof}
We can now pass to the proof of Theorem \ref{MAIN} following the strategy outlined at the beginning of the section and performed explicitly in the case of dual EPW sextics. The main ingredients of the proof in the general case have already been introduced in Sections \ref{section r cyclic} and \ref{section isometries of transcendental}. 

Take $X$, $Y$ be two \HK{} fourfolds as in the assumptions and let $\phi: T_X \to T_Y$ be a Hodge isometry. Let us extend $\phi$ to a map $$\bar{\phi}: H^2(X,\mathbb Q)\to H^2(Y,\mathbb Q)$$ by setting $\bar{\phi}(L_X)=L_Y$ and observe that $\bar{\phi}$ is a rational Hodge isometry. 

Let us first assume that we are in the setting of item (1) of Theorem \ref{MAIN}, i.e. $L_X$ is of divisibility 1 and degree $2d$, where $$d\equiv 0, 1 \text{ or } 5\ \mod\ 8 .$$ Then by Proposition \ref{proof-lem4.8} item (2) we can decompose 
$$\bar{\phi}= \eta_Y\circ \phi_1 \circ \phi_0\circ \eta_X^{-1},$$
where $\phi_0$ and $\phi_1$ are equivalent up to composing on both sides by automorphisms of $\Lambda$ with reflections $\rho_{r_0}$ and $\rho_{r_1}$. It follows by Torelli theorem that there exists a variety $Z$ of Picard number 1 polarized by $L_Z$ and markings 
$$\eta_{Z,0}: H^2(Z,\mathbb Z) \to \Lambda,$$
$$\eta_{Z,1}: H^2(Z,\mathbb Z) \to \Lambda,$$
$$\eta_X: H^2(X,\mathbb Z)\to \Lambda,$$
$$\eta_Y: H^2(Y,\mathbb Z)\to \Lambda.$$

such that for some triples $(r_i,s_i,l_i)$ with $r_i$, $s_i$ coprime $\frac{2l_i}{r_i}$ odd, the following hold:
\begin{enumerate}
    \item $\eta_{Z,0}(L_Z)=L_{r_0,s_0,l_0}$, $\eta_Y(L_Y)=L'_{r_0,s_0,l_0}$,
    \item  $\eta_{Z,1}(L_Z)=L'_{r_1,s_1,l_1}$ and $\eta_X(L_X)=L_{r_1,s_1,l_1}$,
    \item $\eta_Y^{-1} \circ \rho_{r_0} \circ \eta_{Z,0}$ and $\eta_{Z,1}^{-1} \circ \rho_{r_1} \circ \eta_X$ are rational  Hodge isometries.
\end{enumerate}
It then follows by Lemma \ref{proof-div1} item (1) that we have derived equivalences $D^b(Y)=D^b(Z)$ and $D^b(X)=D^b(Z)$. We conclude that $X$ and $Y$ are derived equivalent.



To compute the number of Fourier-Mukai partners for $X$ as in the considered case (1) of Theorem \ref{MAIN} observe that the just proven condition combined with \cite[Corollary 9.3]{B} implies that $Y$ is a Fourier-Mukai partner of $X$ if and only if $T_Y$ is Hodge isometric to $T_X$. Moreover, by Torelli theorem  if the Hodge isometry between $T_X$ and $T_Y$ extends to a Hodge isometry between $H^2(X,\mathbb Z)$ and $H^2(Y,\mathbb Z)$ then $X$ and $Y$ are isomorphic. It follows that the number of Fourier-Mukai partners of $X$ is equal to the number of automorphisms of the lattice $T_X$ up to restrictions of automorphisms of  $H^2(X,\mathbb Z)$ preserving $T_X$. These by  Proposition \ref{proof-lem4.8} and the discussion in Section \ref{section isometries of transcendental} are in one to one correspondence with automorphisms of the discriminant preserving the induced quadratic form. The latter correspond to the number of solutions of the equation $$a^2\equiv 1\ \mod\ 4d$$ for $a\in \{0,\dots,2d-1\}$, which is known to be $2^{\tau(d)}$ as these solutions correspond to decompositions of $d$ into a product of coprime numbers.

The remaining cases are done in a completely analogous way using other items of Proposition \ref{proof-lem4.8} or \ref{proof-div2} together with Lemma \ref{proof-div1} or \ref{proof-lem-div2}. For the number of Fourier-Mukai partners in the case of divisibility 2 the analogous argument leads to the number of solutions of the equation $$a^2\equiv 1\ \mod d\ \ \text{for}\ a\in \{0,\dots,d-1\}.$$ Since whenever $L_X$ is of divisibility 2 we have $d$ odd, this number of solutions is again $2^{\tau(d)}$.

\section{Twisted derived equivalences of abelian fibrations}\label{EPW4}
In order to prove twisted derived equivalences of \HK{} fourfolds with Picard rank $2$ we have to study Fourier-Mukai kernels of rank $0$.
They are more difficult to treat because we cannot use the results from \cite{Mar1} and deform them as hyperholomorphic bundles.
Indeed, consider a pair of \HK{} fourfolds with Picard lattices $U(2)$ and $2\oplus (-2)$ having the $-2$ class of divisibility $2$ in the transcendental lattice.
Then those manifolds have isomorphic transcendental lattices with discriminant $\ZZ_2^3$. The corresponding automorphism of the discriminant is given by the matrix
$$\left( \begin{array}{ccc}
     1&1& 0 \\
     0&1&1  \\
     0&0&1
\end{array}\right).
$$
We cannot describe this automorphism as a specialisation of a reflection so we cannot construct the derived equivalence as a specialisation of an equivalence between \HK{} fourfolds of Picard rank $1$. 
Let us first introduce some useful tools.

\subsection{Twisted fibrations}\label{fibration}

Let us explain how to construct a moduli space of twisted sheaves on a K3 surface $S$ as a twistor over a relative Picard scheme of a family of curves. We consider the case of degree $2$ surfaces for clarity.
 Let $p\colon \C\to \PP^2$ be the universal family of genus $2$ curves on $S$.
 Recall that the relative Jacobian fibration $X_0=\operatorname{Pic}_0(\C) \to \PP^2$ is an abelian group scheme and admits a structure of a Lagrangian fibration.
 As observed in \cite{S}, the variety $X_0$ is isomorphic to $M_{(0,h,1)}(S)$, thus it is a \HK\ fourfold with Picard lattice $2\oplus (-2)$ with the $-2$ class of divisibility $2$ (it is also isomorphic to the square $S^{[2]}$). 
Denote by $X_1$ the relative fibration $\operatorname{Pic}_1(\C) \to \PP^2$. This fibration has no sections 
 and $X_1$ is a \HK\ fourfold with Picard lattice that is isometric to the hyperbolic plane $U$.

As observed in \cite[Example 7.8]{Mar3}, $X_1$ is a torsor over $X_0$. 
Note that, more generally, torsors over $X_0$ are parameterised by torsion classes in $H^1(\PP^2,X_0)$. 
Moreover, as in \cite[Eq.~7.7]{Mar3} there is a natural map \begin{equation}\label{epo} \varphi\colon H^1(\PP^2,X_0)\to H^2(\oo_{S}^{\ast}),\end{equation}
induced by the sequence $0\to \ZZ\to \oo_{\C}\to \oo_{\C}^{\ast}\to 0$.
From \cite[Example 7.8]{Mar3}, this map, in the considered case of  $2$ torsion classes,  has a cyclic kernel of order $2$.

The following proposition shows that one can look both at double EPW sextics with Lagrangian fibrations and double EPW quartics as torsors over $X_0$ (also as moduli spaces of twisted sheaves).
Note that it can be  generalised to higher degree $K3$ surfaces $S$ and Brauer classes of any order as will be illustrated in Section \ref{analogADM}.
\begin{prop} Each $2$-torsion class $\gamma$ in $H^1(\PP^2,X_0)_2$ gives a torsor over $X_0$ 
isomorphic to a moduli space of twisted sheaves on  the $K3$ surface $S$ of degree $2$ with Brauer class $\beta=\varphi(\gamma)$ in $H^2(\oo^{\ast}_{S})_2$ and some twisted Mukai vector.
\end{prop}
\begin{proof}
If $\beta$ is a $2$-torsion Brauer class on $(S,h)$ and $B$ is its $B$-lift then $M_{(0,h,c)}(S,B)$ for any $c\in \ZZ$ is naturally a torsor over $X_0$. Indeed, first observe that $M_{(0,h,c)}(S,B)$ admits a projection $\pi_c$ to $\mathbb P^2=|C|$ by associating to a twisted sheaf its support. It is now enough to observe that the operation of tensor product of a twisted sheaf on $S$ corresponding to a point from $M_{(0,h,c)}(S,B)$  
with a rank $0$ sheaf corresponding to a point from $M_{(0,h,1)}(S)$. Having the same support it is well defined and gives rise from Lemma \ref{relPic} to  a twisted sheaf on $S$ corresponding to a point from $M_{(0,h,c)}(S,B)$ with the same support. Moreover, for two twisted sheaves with the same support corresponding to points from $M_{(0,h,c)}(S,B)$ one is obtained from the other by tensoring with a rank 0 sheaf corresponding to a point $M_{(0,h,1)}(S)$.
Note now that $$M_{(0,h,c)}(S,B)\simeq M_{(0,h,c+2)}(S,B)$$ for each Brauer class $\beta$ and its $B$-lift $B$. Consequently we have exactly two torsors  which are moduli spaces of twisted sheaves twisted with respect to the given Brauer class $\beta$.

From Lemma \cite[Lemma 7.3]{Mar3} $$Br(S)=Br(\C)=H^1(\PP^2,R^1p_{\ast}\oo_{\C}^{\ast}).$$
Thus a Brauer class on $S$ can be represented as a gerbe $\Gamma$ which itself corresponds to a covering $\{ U_i \}$ of $\PP^2$ with $U_{ij}=U_i \cap U_j$ equipped with line bundles $L_{ij}$. Now, let $\C_i$ and $\C_{ij}$ defines the preimages of $U_{i}$ and $U_{ij}$ in $\C$. By possibly densifying the cover assume that $\C_{i}\to U_i$ admits a section $\zeta_i$.  In that case, any family of $\beta$ twisted sheaves on $S$ gives rise to a family of $\Gamma$-twisted sheaves on $\C$. Such a family is a collection  $\{\mathcal F_i\}$ of sheaves $\mathcal F_i$ on $\C_{i}$ such that $\mathcal F_i|_{U_{ij}}=\mathcal F_j|_{U_{ij}\otimes L_{ij}}$. By multiplying the sheaves by suitable powers of $\mathcal O_{\C_i}(\zeta_i (U_i))$ we get bundles of relative degree 0.  In this way we get isomorphisms $$h_i: \pi_c^{-1}(U_i)\to \Pic_0(\C_i).$$ It follows that the torsor $M_{(0,h,1)}(S,B)$, for which,  by Lemma \ref{relPic} the universal bundles $\mathcal F_i$, are relative degree 0 represents the Brauer class $\beta$. Considering $M_{(0,h,1)}(S,B)$,  by copying the argument from \cite[Example 7.8]{Mar3},  we get a different torsor, but whose class is mapped by $\varphi$ to the same Brauer class $\beta$.
\end{proof}

We can understand the map $\varphi$ more precisely and find the twisted Mukai vector $(0,h,c)$ of the moduli space that is constructed as a torsor.
In order to apply the next Lemma we need to find the rank $2$ twisted vector bundle $\U$ such that $\PP(\U)$ is a Brauer-Severi variety that represents a given $2$ torsion Brauer class. We can do this by applying \cite[Table 1 with Proposition 2.7]{vGK} and using the computation \cite[Lemma 6.2]{K} of the degree of $\PP(\U)$ along $C$.
\begin{lem}\label{relPic}  Let $(S,h)$ be a polarised K3 surface. Consider a $B$-twisted sheaf $T$ for a Brauer class $\beta\in H^2(\oo_S^{\ast})_2$ with $B$-lift $B$ that is supported on a smooth curve $C\in |h|$.
Let $\U$ be a rank two $B$-twisted vector bundle 
 such that the bundle $T\otimes \U^{\vee}|_C=A$ is a vector bundle of rank $2$ on the curve $C$  
(\cite[\S 1.1]{Y}). 
Then the twisted Mukai vector of $T$ is the following 
$$v_B=(0,h,\frac{1}{4}(8-2(c_1(A)+K_C))).$$ 

\end{lem}
\begin{proof}

Let $\alpha \in Pic_d (C)$ where $C\subset S$ is a hyperplane section of genus $g$ of a K3 of genus $g$. Then $\alpha$ can seen as a sheaf on $S$ with Mukai vector $ch(\alpha)=(0, h, d+1-g)$. 

Following the notation from \cite[Definition 3.1]{Y}. we denote by $p:Y\to X$ the projective bundle with $G$ the natural bundle on $Y$ (as in \cite[\S 2.2]{Y})
such that $Y=\PP(\U)$ for the twisted sheaf $\U$.
Let $F$ be the $Y$ sheaf corresponding to $T$ through \cite[(1.1)]{Y}.
We have 
\begin{center}
\begin{equation}\label{opl}
v_G=\frac{ch(R p_{\ast}(G^{\vee} \otimes F))}{\sqrt{ch(R p_{\ast}(G\otimes G^{\vee}))}} \sqrt{td(X)}.
\end{equation}
\end{center}
Note that the twisted Mukai vector $v_B$ is equal to $v_G$ for sheaves of rank $0$ \cite[Remark 3.2]{Y}.
We use formulas about Chern classes of push-forwards \cite[15.3.4]{F}.
Then $$ch(R p_{\ast}(G^{\vee} \otimes F))=ch(A).$$
We have $\sqrt{td(X)}=1+p$ with $p$ a point on the $K3$ surface.

In the case $ch^B(\U)=(2,2B,s)$ we have $$\sqrt{ch(R p_{\ast}(G\otimes G^{\vee})}=\sqrt{ch(\U\otimes \U^{\vee})}=(2,0,s-B^2).$$
With the notation of \cite[15.3]{F} we have for $E=A$ and $c_1(L)=-K_C$ the following
$$P(L^{\vee},E)=2-c_1(L)+c_1(E).$$ We infer from \cite[Theorem  15.3]{F} that $c_2(f_{\ast}E)=c_1(E)-c_1(L)$ where $f\colon C \to S$ is the inclusion. Thus the Mukai vector
$$ch(A)=(0,2h,\frac{1}{2}(8-2(c_1(E)-c_1(L)))).$$
 We conclude by Equation \ref{opl} that $v_G=(0,h,\frac{1}{4}(8-2(c_1(E)-c_1(L))))$.
\end{proof}


\subsection{Derived equivalences along Lagrangian fibrations}\label{analogADM}
In this section, we discuss the second construction of derived equivalences working for \HK{} fourfolds of Picard rank $ 2$. Using this construction, we obtain moreover, many twisted derived equivalences between \HK{} fourfolds and in consequence, we
generalise the results of \cite{S} and \cite{ADM}.
\begin{thm}\label{Main}
Suppose that two hyper-K\"ahler fourfolds are isomorphic to moduli spaces $M_{(0,h,a)}(S,B)$ and $M_{(0,h,b)}(S,B)$ of twisted sheaves on the same twisted
$K3$ surface $(S,\beta)$ such that $S$ has Picard rank $1$. Then they are twisted derived equivalent.
\end{thm}
The theorem will be proved at the end of the section.
Let $V=H^0(\oo_S(1))$ and $$\C=\{ (x,C)\in S\times \PP \ V \colon x\in C \}$$ be the universal curve and the projection map $\C \to \PP V$. 
Let us fix a Brauer class with a $B$-lift $B$ on $S$ or equivalently on $\C$ (see Section \ref{fibration}).
Denote by $\M_d\to \PP(V)$ the moduli space of twisted sheaves $M_{(0,h,a)}(S,B)$ together with the natural map associating the support to a sheaf.
This can be seen as a relative twisted Picard variety of degree $d$ line bundles on the fibers (cf.~it is called $\Pic_B^d(\C)$ in \cite[Theorem 1.1]{HM}). The relation between $d$ and the Mukai vector $(0,h,a)$ is discussed in Lemma \ref{relPic}.
We have the following diagram:
\[ \xymatrix{
& \C \times_{\PP V} \M_m \times_{\PP V} \M_n \ar[ld]_{p_{12}} \ar[d]_{p_{13}} \ar[rd]^{p_{23}} \\
\C \times_{\PP V} \M_m \ar[d]_{q_1} \ar[rd]_<{q_2}
& \C \times_{\PP V} \M_n \ar[ld]^<{r_1}|\hole \ar[rd]_<{r_2}|\hole
& \M_m \times_{\PP V} \M_n \ar[ld]^<{u_1} \ar[d]^{u_2} \\
\C \ar[rd]_{v_1} & \M_m \ar[d]_{v_2} & \M_n \ar[ld]^{v_3} \\
& \PP V
} \]

Let us consider an \'etale covering $\{ U_i \}$ of $\PP V$ such that $\C_i\to U_i$ have sections $s_i$ where $\C_i=\C|_{U_i}$. Denote by $\M^i_d\subset \M_d$ the open set corresponding to $U_i$. For $U_{ij}=U_i\cap U_j$ we have the restrictions $\Pic_0^{ij}$, $\M_n^{ij}$ and $\C_{ij}$. 
We claim that $\M_n$ are twistor spaces over $\Pic_0\to \PP(V)$ (in fact over $\Pic_0'$ the set of smooth fibers).
This means that there are isomorphisms $\varphi_i\colon Pic_0^i\to \M_n^i$.
We need to prove that there exist sheaves $\alpha_{ij}$ on $\C_{ij}$ that are relatively degree $0$ and such that the composition $$\Pic_0^{j}\to \M_n^i\to Pic_0^j$$ of $\varphi_i$ and $\varphi_j^{-1}$ over $U_{ij}$ is relatively a translation by $\alpha_{ij}$. For example $\Pic_1$ is a torsor with $\alpha_{ij}=\oo_{\C}(s_i-s_j)$.

Let $\mathcal J_i$ be universal families over $\C_i \times Pic_0^i$. 
Their pullbacks to $\M_n^i\times \C_i$ are denoted by $\F_{n}^i$.
Consider
\begin{multline*}
P_{mn}^i = (\det p_{23*} (p_{12}^* \F_{m}^i \otimes p_{13}^* \F_{n}^i))^{-1}
\otimes u_1^* v_2^* (\det v_{1*} \oo_{\C})_i^{-1} \\
\otimes u_1^*(\det q_{2*} \F_{m}^i) \otimes u_2^*(\det r_{2*} \F_{n}^i).
\end{multline*}
\begin{prop}\label{twistedFM} $P_{mn}^i$ define an $\alpha_n^{-m}\times \alpha_m^{-n}$ twisted sheaf over $\M_n\times \M_m' \cup \M_n'\times
\M_m$ where $\M_n'$ and $\M_m'$ are sums of the smooth fiber of Lagrangian fibrations.
\end{prop} 

\begin{proof}

Note that $\mathcal F^j_m=\mathcal F^i_m \otimes q_1^*\mathcal L_{ij}^{\mathcal C,m}\otimes q_2^*\mathcal L_{ij}^{\mathcal M,m}$ for some line bundles $\mathcal L_{ij}^{\mathcal C,m}$ of fiberwise degree 0 on $\mathcal C_{ij}$ and $\mathcal L_{ij}^{\mathcal M,m}$ on $\mathcal M_{ij}$. Similarly 
$$\mathcal F^j_n=\mathcal F^i_n \otimes r_1^*\mathcal L_{ij}^{\mathcal C,n}\otimes r_2^*\mathcal L_{ij}^{\mathcal M,n}$$ for some line bundles $\mathcal L_{ij}^{\mathcal C,n}$ of fiberwise degree 0 on $\mathcal C_{ij}$ and $\mathcal L_{ij}^{\mathcal M,n}$ on $\mathcal M_{ij}$.

Let us now compare $P^i_{mn}$ with $P^j_{mn}$ over $\mathcal C^{ij}$, i.e. in $\mathcal M^{ij}_m\times_{\mathcal C^{ij}} \mathcal M^{ij}_n$.
We have
\begin{multline*}
P_{mn}^j = (\det p_{23*} (p_{12}^* (\F_{m}^i\otimes q_1^*\mathcal L_{ij}^{\mathcal C,m}\otimes q_2^*\mathcal L_{ij}^{\mathcal M,m}) \otimes p_{13}^* (\F_{n}^i \otimes r_1^*\mathcal L_{ij}^{\mathcal C,n}\otimes r_2^*\mathcal L_{ij}^{\mathcal M,n})))^{-1}\\
\otimes u_1^* v_2^* (\det v_{1*} \oo_{\C})_i^{-1}
\otimes u_1^*(\det q_{2*} (\F_{m}^i\otimes q_1^*\mathcal L_{ij}^{\mathcal C,m}\otimes q_2^*\mathcal L_{ij}^{\mathcal M,m})) \otimes u_2^*(\det r_{2*} (\F_{n}^i\otimes r_1^*\mathcal L_{ij}^{\mathcal C,n}\otimes r_2^*\mathcal L_{ij}^{\mathcal M,n}))=\\
(\det  (p_{23*} (p_{12}^* (\F_{m}^i\otimes q_1^*\mathcal L_{ij}^{\mathcal C,m})\otimes p_{13}^* (\F_{n}^i \otimes r_1^*\mathcal L_{ij}^{\mathcal C,n}))\otimes u_1^*\mathcal L_{ij}^{\mathcal M,m}\otimes u_2^* \mathcal L_{ij}^{\mathcal M,n} ))^{-1}
\otimes u_1^* v_2^* (\det v_{1*} \oo_{\C})_i^{-1} \\
\otimes u_1^*(\det (q_{2*} (\F_{m}^i\otimes q_1^*\mathcal L_{ij}^{\mathcal C,m})\otimes \mathcal L_{ij}^{\mathcal M,m})) \otimes u_2^*(\det r_{2*} (\F_{n}^i\otimes r_1^*\mathcal L_{ij}^{\mathcal C,n})\otimes \mathcal L_{ij}^{\mathcal M,n})=
(\det  (p_{23*} (p_{12}^* (\F_{m}^i \\ \otimes q_1^*\mathcal L_{ij}^{\mathcal C,m})\otimes p_{13}^* (\F_{n}^i \otimes r_1^*\mathcal L_{ij}^{\mathcal C,n}))))^{-1}\otimes (u_1^*\mathcal L_{ij}^{\mathcal M,m}\otimes u_2^* \mathcal L_{ij}^{\mathcal M,n} ))^{-\rho_1}  \otimes u_1^* v_2^* (\det v_{1*} \oo_{\C})_i^{-1} \\
\otimes u_1^*(\det (q_{2*} (\F_{m}^i\otimes q_1^*\mathcal L_{ij}^{\mathcal C,m})))\otimes u_1^*(\mathcal L_{ij}^{\mathcal M,m})^{\rho_2} \otimes u_2^*(\det r_{2*} (\F_{n}^i\otimes r_1^*\mathcal L_{ij}^{\mathcal C,n}))
\otimes u_2^*(\mathcal L_{ij}^{\mathcal M,n})^{\rho_3}.
\end{multline*}
Here, $$\rho_1=\rk  (p_{23*} (p_{12}^* (\F_{m}^i\otimes q_1^*\mathcal L_{ij}^{\mathcal C,m})\otimes p_{13}^* (\F_{n}^i \otimes r_1^*\mathcal L_{ij}^{\mathcal C,n})) $$
$$\rho_2=\rk (q_{2*} (\F_{m}^i\otimes q_1^*\mathcal L_{ij}^{\mathcal C,m}))$$
$$\rho_3=\rk  (r_{2*} (\F_{n}^i\otimes r_1^*\mathcal L_{ij}^{\mathcal C,n}))$$

which can be computed explicitly from the Riemann-Roch theorem. In particular,
$$\rho_1=\rk  (p_{23*} (p_{12}^* \F_{m}^i\otimes p_{13}^* \F_{n}^i) =m+n+1-g$$
$$\rho_2=\rk (q_{2*} \F_{m}^i)=m+1-g$$
$$\rho_3=\rk  (r_{2*} \F_{n}^i)=n+1-g.$$
We conclude $\rho_1-\rho_2=-n$ and $\rho_1-\rho_3=-m$.

On the other hand the bundle
\begin{multline*}
\det  (p_{23*} (p_{12}^* (\F_{m}^i\otimes q_1^*\mathcal L_{ij}^{\mathcal C,m})\otimes p_{13}^* (\F_{n}^i \otimes r_1^*\mathcal L_{ij}^{\mathcal C,n}))))^{-1}\otimes\\ u_1^*(\det (q_{2*} (\F_{m}^i\otimes q_1^*\mathcal L_{ij}^{\mathcal C,m})))\otimes u_2^*(\det r_{2*} (\F_{n}^i\otimes r_1^*\mathcal L_{ij}^{\mathcal C,n}))
\end{multline*}
is just obtained by translation of $P^i_{mn}$ by $L_{ij}^{\mathcal C,n}\boxtimes L_{ij}^{\mathcal C,n}$ understood as an element of $\Pic_0\mathcal C_i\times_{\mathcal C_i} \Pic_0\mathcal C_i $, but $P^i_{mn}$ is homogeneous, i.e. invariant under translations along fibers.
We conclude that:

$$P_{mn}^j=P_{mn}^i\otimes (\mathcal L_{ij}^{\mathcal M,m})^{-n}\otimes  (\mathcal L_{ij}^{\mathcal M,n})^{-m}.$$
\end{proof}

\begin{proof}[Proof of Theorem \ref{Main}]
We consider the push forward a Poincare bundle
$\bar{P}_{mn}$ on $\M_n\times \M_m$. Note that it is no more a line bundle on $\M_n\times_{\PP^2} \M_m$ but a CM sheaf whose fibers jump along the singular set of the fibers. We can however act similarly to \cite[\S 2]{ADM} and prove the following claim which implies directly Theorem \ref{Main}.
\begin{claim}\label{main} If the Picard rank of $\M_n$ (and $\M_m$) is $2$ then the categories $D(\M_m,\alpha_m^n)$ and $D(\M_n,\alpha_n^m)$ are equivalent.

\end{claim}

The Claim follows since $\bar{P}_{mn}$ defines a twisted Fourier-Mukai functor as seen analogously to \cite[Theorem  C]{A}. 
The assumption on the Picard rank of $S$ is needed as \cite{A} consider Jacobians of integral curves. 
\end{proof}


\subsection{Brauer class of universal families}
In order to apply the results from the previous section we need to understand the Brauer classes $\alpha_n$ considered in Proposition \ref{twistedFM}. This is the aim of this section.
We shall discuss the special case of hyper-K\"ahler fourfold $X$ of $K3^{[2]}$-type admitting a BN contraction $X\supset E\to S$ (this is a contraction such that the fiber of $E$ is a $-2$ class of divisibility $1$). 
In this case, $X$ is the moduli space of sheaves on $S$, so from \cite[Table 1]{vGK} there are four types of such contractions.
In each case it is proved that $X$ is isomorphic to the moduli space $M_v(S,B)$ (for some Brauer class $B$ and twisted Mukai vector $v=(2,.,.)$) of twisted sheaves of rank $2$ on the K3 base $S$ of the contraction. Let us denote by $\beta$ the Brauer class on $M_v(S,B)$ that gives the obstruction to the existence of the universal twisted sheaf on $M_v(S,B)\times S$.
Note that the obstruction class in the case of moduli spaces of twisted sheaves is  more difficult to find compared to the case of untwisted moduli spaces. Such description in the general case is a challenging problem.  
\begin{prop}\label{universal} The Brauer class $\beta \in Br( M_v(S,B))$ giving the obstruction to the existence of the universal sheaf on $S\times M_v(S,B)$ is either trivial or has a $B$-lift $\frac{\delta}{2}$ with $\delta$ a $-2$ class of divisibility $2$.

\end{prop}
\begin{proof} 
The projectivisation of the universal family $\mathcal U$ can be seen as a conic bundle over $S\times M_v(S,B)$.
From \cite[\S 1]{vGK} $q\colon E\to S $ is a conic bundle over $S$ with Brauer class $C$ (such that the pull-back of $C$ on $E$ is an integral class).

We first describe the restriction $\mathcal U_p$ of the bundle $\mathcal U$ to $E$ over a general point  $p\in S$.
We shall prove that $\mathcal U_p$ is a vector bundle of rank $2$ such that the restriction of the Brauer class $\alpha_0$ to $E$ is trivial. 

We construct explicitly the universal family over $S\times E$. 
Indeed, denote by $\E$ the twisted sheaf such that $E=\PP(\E)$. Then we have $$\PP(\E)\xleftarrow{p} \PP(\E)\times S\xrightarrow{\pi} S.$$ 
We consider $\pi^{\ast} (\E)$ and the twisted line bundle $\oo_E(1)$ with trivial Brauer class.
Now consider the natural embedding $\bar{E}\subset E\times S$ such that $$\bar{E}=\{ (P,q(P))\in E\times S  \}.$$ The map $\pi^{\ast} (\E) \to \oo_{\bar{E}}(1)$ is given by the diagonal of the fiber product $E\times_{S} E$ (indeed a point $Q$ over $P$ on $E$ in the fiber of $\PP(\pi^{\ast}\E)$ on $E$ gives a map $\pi^{\ast}\E\to \oo_P\to 0$). Thus the restriction of the universal family to a fiber over $P\in S$ is an elementary transform of $2\mathcal{O}_E$ at a point.

Finally, let $B$ be the Brauer class on $X$ giving the obstruction to the existence of the twisted universal family on $S\times X$.
Let us prove that it is trivial.
Let $e$ be the class of a fiber on $E$. We deduce as in \cite[Lemma 6.2]{K} that $e.B=0$ (for some $B$-lift $B$) since the class of $e$ in $H_2(X,\ZZ)\subset H^2(X,\ZZ)$ is $E$.
In fact since $B$ is represented by a $\PP^1$ bundle it follows that $B.e$
is even.

From the universal property we infer that the  restriction of the universal twisted family on $S\times X$  to $S\times E$ is the above constructed family.
Consider the injective composed map 
$$r: H^2(X,\ZZ)\supset E^{\perp}\to \{A \in H^2(E,\ZZ)| \ A.f=0  \} \subset H^2(E,\mathbb Z)\to H^2(S,\ZZ).$$

From \cite[\S 4.3]{vGK} the image is of index $2$ characterised as the classes that have integer intersection with $C$ and moreover, $$\{A \in H^2(E,\ZZ)| \ A.f=0  \}=q^*H^2(S,\ZZ).$$
Next, $ H^2(X,\mathbb C)\to H^2(E,\mathbb C)$ is an isomorphism of Hodge structures that induces an embedding $$\varphi \colon H^2(X,\ZZ)\to H^2(E,\ZZ).$$ 
We claim that we can choose a $B$-lift $R\in H^2(X_0,\ZZ/2)$ of $B$ such that $R \in T_{X_0}\otimes \QQ$ and such that the image $\varphi(2R)=2t\in  H^2(E,2\ZZ)$. Indeed, suppose that for a $B$-lift $T$ we have $$\varphi(2T)=2t+K_E+\bar{H}\in H^2(E,2\ZZ)+\Pic(E)$$ where $\bar{H}$ is a pull back of a divisor from $S$ to $E$ (this holds because $T$ restricted to $E$ is the trivial Brauer class). Then $2T=2R+E+H$ such that $H$ is in $E^{\perp}\subset \Pic(X)$ and induces $\bar{H}$ on $S$. Then $$2R\in T_X,\ \varphi(2R)=2t$$ and $R$ is another $B$-lift of $T$. This prove the claim.

Let $g\in H^2(S,\ZZ)$ be such that $q^{\ast}(g)=t$.
Now if $g\in \im(r)$ then $r(R)=g$ since $r$ is an injection so the Brauer class of $R$ is trivial.
If $g$ is not in $\im(r)$ then it is a combination from the proof of \cite[Proposition 4.4]{vGK} with the notation there that $g=\frac{1}{2}r(\delta-e)+d$ 
where $d\in \operatorname{im}(r)$ ($g=v^*+s^*$ modulo $\operatorname{im}(r)$ in the notation from \cite[Proposition 4.4]{vGK} where $$r(\delta-e)=r(v-2v^{\ast}-s+2s^{\ast}) = 2(v^\ast +s^{\ast})$$
thus $\frac{\delta-e}{2}$ is a $B$-lift of the Brauer class of $B$ but $e\in \Pic(X)$.
\end{proof}
In the case of double EPW sextics admitting a BN contraction we shall describe the Brauer class explicitly in Proposition \ref{uni}.


\subsection{Duality between double EPW sextics and double EPW quartics}\label{EPWquartic}

Let us now illustrate the new constructions in the case discussed at the beginning of the section of a pair of \HK{} fourfolds of $K3^{[2]}$-type with Picard lattices $U(2)$ and $(2)\oplus (-2)$; so a pair of a very general double EPW quartic and a special double EPW sextic. Those examples have in general Picard rank $2$ and, as was discussed before, we cannot obtain the derived equivalence using Claim \ref{main}.

Recall that a double EPW quartic is a $K3^{[2]}$-type manifold with Picard lattice containing $U(2)$. In this case the map given by the sum of generators gives a $2:1$ map to a special quartic section of the cone over the Segre product $\PP^2\times \PP^2$ in $\PP^9$, see \cite{IKKR}. A double EPW sextic is a \HK{} fourfold of $K3^{[2]}$-type with polarisation of degree $2$. The polarisation gives in general a $2:1$ map to a special sextic hypersurface in $\PP^5$ called an EPW sextic (see \cite{O1}).   

Let $(S,h)$ be a very general K3 surface of degree $2$.
The polarisation $h$ of degree $2$ gives a $2:1$ map $S\to \PP^2$ branched along a curve of degree $6$.
We consider $X_i=M_{v_i}(S,B)$, the moduli spaces of twisted sheaves with Mukai vectors $$v_0=(0,h,1)\ \text{and}\ v_1=(0,h,0),$$ where $B'$ is a $B$-lift of a corresponding Brauer class on $ T$. Here $B\in H^2(S,\mathbb{Z})$ is a $B$-field lift of a Brauer class $\beta$ 
determined by the property that it has a $B$-lift with the properties $Bh=0$, $B^2=0$. 
It follows from Lemma \ref{relPic} that $X_0$ (resp.~$X_1$) can be seen as the relative $\Pic_0$ (resp.~$\Pic_1$) over the twisted fibration $\C\to \PP^2$ with Brauer class corresponding to $B$ through \cite[Lemma 7.3(4)]{Mar2}.

It follows from \cite[Theorem  5.1]{CKKM} that the fourfold
$X_1=M_{v_1}(S,B)$ is isomorphic to a double EPW quartic.
In particular, it admits two Lagrangian fibrations and we can associate to it a second K3 surface $(T,t)$ of degree $2$ such that
$X_1=M_{(0,t,0)}(T,B)$. Let us denote by $X_0'$ the moduli space of twisted sheaves $M_{(0,t,1)}(T,B')$.  

\begin{lem}
The fourfold $X_0$ is a double EPW sextic with Picard lattice $(2)\oplus (-2)$ such that the class $(-2)$ is of divisibility $1$ (it is a class of a BN contraction). The fourfold $X_0'$ is the dual to the double EPW sextic $X_0$. 
\end{lem}
\begin{proof}
It is a double EPW sextic from \cite[Table 1]{vGK}.
The K3 surface $T$ is then isomorphic to $M_{(2,2B+h,1)}(S,B)$ and is the base of the contraction on $X_0'$.
We conclude the proof comparing the periods of $X_0$ and $X_0'$ see \cite[\S 1.6]{O1}.
\end{proof}

From Claim \ref{main} we obtain a universal family on $X_0\times X_1$.  It is defined on tubes, i.e.~open subsets $X_0^i\times_{\PP^2} X_1^i$, by line bundles $P_{01}$ that are fiberwise degree $0$.
Let us show that we can glue them in another way by dividing the charts and obtain a line bundle that is no longer degree $0$ on the fibers (cf.~\cite[\S 1]{Y}).
In this way we refine the result of Claim \ref{main} by computing the Brauer twist of the universal family in this case.
\begin{prop}\label{uni}
There exists a universal family on $X_0\times X_1$ which is a sheaf supported on $X_0\times_{\PP^2}X_1$ such that it is fiberwise a translation of the Theta divisor when restricted to the abelian fibers on $X_1$ and restricted to the abelian fibers on $X_0$. 
\end{prop}
\begin{proof}
    First let us construct a line bundle on $$D_f=f\times_{\PP^2}X_1\subset X_0\times_{\PP^2}X_1$$ where $f$ is an exceptional fiber on $E$.
     The fiber $f$ corresponds to a point $P$ in $S$. Moreover, from \cite{vGK} a point on $f$ corresponds to a line in the quadric fiber over $P$ in the Verra fourfold $V$ corresponding to $X_0$ (such that $X_1$ is the base of a $\PP^1$ fibration on the Hilbert scheme of $(1,1)$ conics contained in $V$). From \cite{IKKR} the point $P$ gives also a line in $\PP^2_2$. The preimage of this line through $\pi_2$ on $D_f$ splits
    into two divisors. Indeed, for a fixed line through $P$ we consider the del Pezzo $K_4$ surface of degree $4$ in $V$ obtained by fixing a line in the image of the other projection on $V$. The $(1,1)$ conics contained in $K_4$ split in two types according to which line in the preimage of $P$ on $K$ they intersect. We denote by $L_f\subset D_f$ the divisor of $(1,1)$ conics that intersect the line from the ruling of $f$ on $V$. In fact $D_f=D_{\bar{f}}$ where $\bar{f}$ is the fiber mapping to the same point through the cover $S\to \PP^2$.

   We define in this way a line bundle on the sum $$
   \bigcup_{f\subset E} L_f\subset  \bigcup_{f\subset E} f\times_{\PP^2} D_f\subset X_0\times_{\PP^2}X_1.$$
We obtain a bundle $\mathcal{L}$ on an open subset of $E\times_{\PP^2} X_1$ induced by the following incidence:
$$\{ (l,c)| \ l \in Hilb_{(1,0)} V, \ c\in Hilb_{(1,1)} V, \ l \cap c \neq \emptyset   \}.$$
  In fact it is naturally defined on smooth fibers and we extend it by flatness.
  
  In order to extend it to the whole product $X_0\times_{\PP^2}X_1$
  we consider the projective duality between the Kummer surfaces $K_1\subset \PP^3$
  and $K_2\subset (\PP^3)^{\vee}$ (see \cite{Cat}). 
  From \cite{IKKR} we can extend this duality to the whole Lagrangian fibration.
  Indeed, let us denote by $Y_0\subset \PP^5$ and $Y_1\subset \PP^9$ the corresponding to $X_0$ and $X_1$ EPW sextic and EPW quartic. 
  After blowing up the singular plane in $Y_0\subset \PP^5$ we infer a fibration of Kummer surfaces in $$\PP(\oo_{\PP^2}(-1)^3\oplus \oo_{\PP^2}).$$
  On the other hand $Y_1$ is described as in \cite[Proposition 3.8]{IKKR} as a dual Kummer surface fibration in the dual bundle $\PP(\oo_{\PP^2}(1)^3\oplus \oo_{\PP^2})$ (i.e.~a small resolution of the cone over the Segre product $C(\PP^2\times \PP^2)$).
   With the notation from \cite{IKKR} suppose that $$A\cap G(3,V_6)={[U_1]}\ \text{and}\ U_1\oplus U_2=V_6.$$
  As observed in \cite[Remark 3.10(1)]{IKKR} the Kummer surfaces are fiberwise dual.
  We observe that the plane $K_1\subset \PP^3\supset \PP_E^2$ that is the base of $E$ is dual to the vertex of the cone over the Segre product naturally contained in the projective space $\PP^3_3\supset K_2$.

  We claim now that the line bundles on $E\times_{\PP^2} X_1$ considered above induce the same projective duality on $Y_0\times_{\PP^2} Y_1$. 
  Indeed, by \cite[Corrolary 9]{Cat} this duality is naturally induced by a correspondence of the shape $u \to \Theta +u$.
Following the construction of $L_f\subset D_f\cap A$ we infer that $L_f$ is a component of the preimage
of a the line $l_P\subset \PP^2_3$. The same line is obtained by the duality of the Kummer surfaces such that a point $P$ in $\PP^2_E$ gives rise to plane in $\PP^3_3$ passing through the vertex and projecting to $l_P$ from this vertex.
We can perform the above in family by pulling back the bundle inducing the duality and infer two choices of a Poincare bundle on an open subset of $X_0\times_{\PP^2} X_1$.

Finally,  extend the line bundle from $E\times_{\PP^2} X_1$ to a line bundle on $X_0\times_{\PP^2} X_1$.
We consider tubes $X_0\times_{U_i} X_1$ where $U_i\subset \PP^2$ is an open cover of the set of bases of smooth fibers of the Lagrangian fibration $X_0\to \PP^2$. 
The duality between the Kummer surfaces is defined globally so it induces a divisor $\bar{D} \subset Y_0 \times_{\PP^2} Y_1$.
For a small enough open subset $U\subset \PP^2$ this effective divisor pull back to $D$ on $X_0\times_{U} X_1$ and splits as a sum of two effective divisors on the central fibers but also on $$X_0\times_{U} X_1\to U$$  we choose a  local section of the fibration  $D\to U$ and consider the component in the fiber intersecting the section. 
We obtain the splitting $D_U^1\cup D_U^2$ of the pull back $D \subset X_0\times_{\PP^2} X_1$ globally by considering the bundle $\mathcal{L}$ on an open subset of $E \times_{\PP^2} X_1$; this gives the choice of the lifting on a general tube $X_0\times_{U} X_1\to U$ 
(on the intersection $U_1\cap U_2$ we choose $D^1_{U_1}$ and $D_{U_2}^!$ such that they agree on $E\times_{U_1\cap U_2} X_1$).



By considering \cite[Theorem A]{A} and acting locally on $\PP^2$, we extend it by pushforward by  obtaining a CM coherent sheaf (no more a vector bundle).
\end{proof}
A direct consequence of the above and Theorem \ref{Main} is the following result.
\begin{corr}
  A very general double EPW quartic is derived equivalent to a double EPW sextic with Picard lattice $(2)\oplus (-2)$ with $(-2)$ of divisibility $1$.
\end{corr}

\subsection{Convolutions of two Fourier-Mukai kernels}
From Proposition \ref{uni} we have two Fourier-Mukai kernels $P_{01}$
and $P_{01}'$ on $X_0\times X_1$ and $X_0'\times X_1$ respectively giving derived equivalences $$D(X_0)=D(X_1)=D(X_0').$$
Using the convolution we infer a rank $8$ sheaf $F_0$ on $X_0\times X_0'$ inducing a derived equivalence $D(X_0)=D(X_0')$.

The  sheaf $F_0$ on $X_0\times X_0'$ is a priori more complicated along a pair of points that are the singular loci of singular fibers of abelian fibrations.
however we can prove the following. Denote by $F_0^p$ the fiber of $F_0$ on $X_0$ for $p\in X_0'$ (resp.~on $X_0'$ for $p\in X_0$).

\begin{lem}\label{VB}
    The sheaf $F_0$ is a  on rank $8$ vector bundle. Moreover, $F_0^p$ is slope-stable and $\E nd(F_0^p) $ hyperholomorphic with respect to all K\"ahler classes on $X_0$ (resp.~on $X_0'$) for each $p\in X_0'$ (resp.~$p\in X_0$).
\end{lem}
\begin{proof} Let us compute the rank of the fiber of $F_0$ at any chosen point $(x,x')\in X_0\times X_0'$. For that let us first recall the notation. 
We consider the  following diagram:
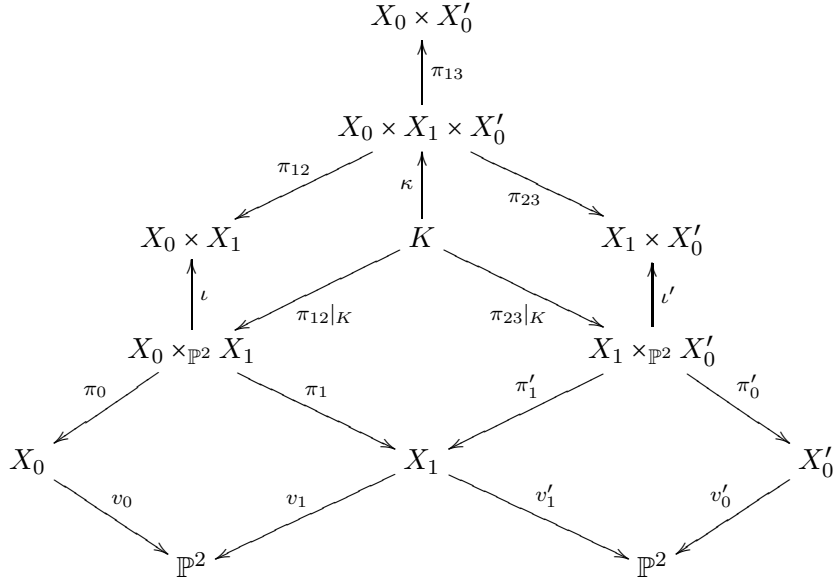
\[ \xymatrix{&&X_0\times X_0' &&\\
&&X_0\times X_1\times X_0' \ar[ld]_{\pi_{12} }\ar[rd]_{\pi_{23} }\ar[u]_{\pi_{13}}&&\\
& X_0\times X_1& K \ar[u]^{\kappa} \ar[ld]^{\pi_{12}|_K }\ar[rd]_{\pi_{23}|_K } & X_1\times X'_0& \\
& X_0\times_{\PP^2}X_1\ar[ld]_{\pi_0} \ar[rd]^{\pi_1}\ar[u]_{\iota}& & X_1\times_{\PP^2}X'_0\ar[ld]_{\pi_1'} \ar[rd]^{\pi_0'}\ar[u]_{\iota'}& \\
X_0 \ar[rd]^{v_0}& & X_1\ar[ld]_{v_1} \ar[rd]^{v'_1}&
&  X_0' \ar[ld]_{v'_0}\\
& \PP^2& & \PP^2& 
}\]

The sheaves $P_{01}$ and $P'_{01}$ are generically rank 0 sheaves on $X_0\times X_1$ and  $X_1\times X_0'$ respectively obtained as pushforwards by the inclusions $\iota$, $\iota'$ Poincare sheaves $\bar{P}_{01}$ and $\bar{P}'_{01}$ of rank 1 defined on the respective fiber products.
Then $$F_0:= {\pi_{13}}_* (\pi_{12}^* P_{01}\otimes \pi_{23}^* P'_{01})$$
and hence the fiber $${F_0}_{(x,x')}\simeq H^0((\pi_{12}^* P_{01}\otimes \pi_{23}^* P'_{01})|_{\pi_{13}^{-1}(x,x')}).$$
Since $\pi_{12}^* P_{01}\otimes \pi_{23}^* P'_{01}$ has support on $K$ we have 
$${F_0}_{(x,x')}=H^0((\pi_{12}^* P_{01}\otimes \pi_{23}^* P'_{01})|_{(\pi_{13}\circ\kappa)^{-1}(x,x') }).$$
Now $${(\pi_{13}\circ\kappa)^{-1}(x,x')}$$ is isomorphic via $\kappa$ to the subscheme of $(\pi_{13})^{-1}(x,x')\simeq X_1$ obtained as the intersection $v_1^{-1}(x)\cap {v'}_1^{-1}(x')$ and under this identification $\pi_{12}^* P_{01}|_{(\pi_{13}\circ\kappa)^{-1}(x,x') }$ is identified with  the sheaf $P_{01}|_{x\times v_1^{-1}(x)}$ which is itself identified with a sheaf $P_x$ of rank 1 on $v_1^{-1}(x)$. It is a compactified Jacobian of a curve of genus 2, such that $P_x$ is represented by a point in the dual compactified Jacobian. Similarly $P'_{01}|_{v_1^{-1}(x')\times x'}$ is a sheaf $P'_{x'}$ on a compactified Jacobian of a curve of genus 2 represented by an element of its dual compactified Jacobian.

Finally, observe that for every $(x,x')\in X_0\times X'_0$ the schematic intersection  $$v_1^{-1}(x)\cap {v'}_1^{-1}(x')$$ is a scheme of length eight. Indeed, by \cite{IKKR} $X_1$ is obtained as a finite double cover of an EPW quartic, which is a section of the cone over $\mathbb{P}^2\times \mathbb{P}^2$ by a quartic not passing through the vertex. The fibers of both Lagrangian fibrations are then preimages by the double cover of the intersection of the EPW quartic with 3-spaces spanned by some of the planes  in $\mathbb{P}^2\times \mathbb{P}^2$ and the vertex of the cone. The intersection of two fibers from different fibrations is then the preimage of the intersection of the EPW quartic with a line passing through the vertex of the cone. Since the EPW quartic does not pass through the vertex of the cone it cannot contain the line hence its intersection with the line is of length 4. The preimage of the scheme of length 4 by the double cover is a scheme of length 8. 

We have now restricted our computation to the following:
$$A=v_1^{-1}(x)\ \text{and}\ A'=v_1^{-1}(x')$$ are two surfaces that are compactified Jacobians of some possibly nodal curves. On $A$ and $A'$ we have two torsion free sheaves $P$ and $P'$ of generic rank 1. The surfaces $A$ and $A'$ are embedded in the fourfold $X_1$ and intersect in a scheme of length 8. We want to compute the dimension of the space of sections of the tensor product of the restrictions of the two sheaves to the intersection $A\cap A'$.

By \cite{A,S} we observe  that both $P$ and $P'$ are pushforwards of line bundles $\tilde P$ and $\tilde P'$ from respective partial normalisations $\tilde A$ of $\tilde{A}'$ of $A$ and $A'$. 
Indeed, for a torsion free sheaf $F$ on $A$ we consider the scheme $Y=\Spec(\mathcal{H}om(F,F))$ that gives a partial normalisation of $A$.
There exists on $Y$ a sheaf $L$ such that $\mathcal{H}om(L,L)=\oo_Y$ whose push-forward is $F$. It is thus enough to consider sheaves that are Cohen- Macaulay and having the properties $\mathcal{H}om(L,L)=\oo_Y$. We can deduce that such sheaves are locally free on varieties that, have transversal nodal or cuspidal singularites (i.e.~$x.y=0$ and $x^2=y^3$). We conclude by considering the possible singularites on $A$ listed in \cite[\S 3.3]{S} using in the remaining case with the cross singularity \cite[Lemma 2.2]{A} in the central point.

Hence we have the following diagram 
$$ \xymatrix{
\tilde{A}\ar[d]_{r} & \ar[l]_{\tilde p} \ar[r]^{\tilde p'}\tilde{A}\times_{X_1} \tilde{A}' \ar[d]_{\pi}&\tilde{A}'\ar[d]_{r'}\\
A\ar[rd] & \ar[l]_{p} \ar[r]^{p'} A\times_{X_1} A'\ar[d] &A' \ar[ld] \\
   & X_1 & } $$
   
Now by the Leray spectral sequence and projection formula

$$
H^0(p^*({r}_*(\tilde P))\otimes {p'}^*({r'}_*(\tilde P)))=H^0(\pi_*(\tilde p^*(\tilde P)\otimes {\tilde {p'}}^{*}(\tilde P')))=H^0(\tilde p^*(\tilde P)\otimes {\tilde {p'}}^{*}(\tilde P'))=8.
$$
Since it is enough to work locally on $X_0^i\times X_1^j$ the sheaf $F_0$ is a vector bundle as a push-forward of a line bundle through a finite map of degree $8$. 


    The stability of $F_0$ follows from the fact that $X_1$ is irreducible.
Indeed, it follows from \cite[p.~35]{Mar1} that it is enough to show that the restriction $F_0^x$ of $F_0$ is stable on  $X_0\times \{ x\}$
for a general $x\in X_0'$ and on $\{y\} \times X_0'$ for $y\in X_0'$.
Next, we find that the restriction of $F_0^x$ to a fiber of the Lagrangian fibration is a sum of different translations of the Theta divisor.
It follows that
a subbundle $E_0^x$ of $F_x^0$ that could destabilise $F_0^x$ restricts to partial sum of the line bundles before. Since each such line bundle corresponds to a point on $X_1$ through the Fourier-Mukai transform, the subbundle induces an analytic subvariety of $X_1$ that is open in the Zariski topology
so gives a subvariety of $X_1$. So $E_0^x=F_0^x$ and $F_0^x$ is stable.

We conclude from \cite[Section 5.6]{Mar1} that $\E nd(F_0^x)$ is hyperholomorphic.
\end{proof}


In order to study the Brauer twist of deformations of the vector bundle $F_0$ we need to compute its first Chern class.
\begin{lem}\label{po}
The Chern class $c_1(F_0^p)$ restricted to a $\PP^1$ fiber on $E$ is of degree $32$ for $p\in X_0'$. 
\end{lem}
\begin{proof}

The fourfold $X_0$ admits a BN contraction (see \cite{vGK}) of a divisor $E$ to a K3 surface. We denote a fiber $\PP^1$ on $E$
by $f$.
Let us compute $c_1(F_0^p)$  
 restricted to $f$. 
 Consider the subset $D\subset X_1$ which is the sum of abelian surfaces corresponding to points $f\subset X_0$.
 It can be seen as the preimage of a conic $C\subset \PP^2$ through the fibration $\pi_1:X_1\to \PP^2_1$ (that is the dual fibration to $\pi\colon X_0 \to \PP^2_1)$.
 The points of $f$ induce fiberwise a line bundle $G_0^f$ on $D$ induced by the universal line bundle $G_0$ on $X_0\times X_1$.
Now denote by $R$ the curve $ D\cap A_p$ where $A_p\subset X_1$ is the corresponding fiber to $p\in X_0'$ and by $C\subset \PP^2_1$ the conic $\pi_1(D)$ such that the induced map $\pi_1|_R\colon R \to C$ is $8:1$.  

We consider the push forward of the bundle $H_0^p$ on $A_p\subset X_1$ (such that $H_0^p$ is the fiber over $p\in X_0'$  of  $H_0$ through the fibration $\pi_1$).  Denote by $L$ the line bundle $G_0^f\otimes H_0^p$. As $F_0$ is the convolution Fourier-Mukai kernels $G_0$ on $X_0\times X_1$ and $H_0$ on $ X_0'\times X_1$. We infer that the restriction of  $F_0^p$ is the push forward of $L$ on $C$ through $\pi_1|_R$.
 
From GRR we find 
that the push forward of $L$ from $R$ to $C$ is of degree  $${\pi_1}_{\ast}c_1(L)= 24+ c_1({\pi_1}_{\ast}(L)).$$

From Lemma \ref{uni} we saw that $L$ is of degree $8$. Indeed, $H_0^p$ is a translation of the Theta divisor on $A_p$ so its image through $\pi_1$ is a line.
Thus the degree of $H_0^p$ is $8$ on $R$.
The line bundle $G_0^f$ is of degree $0$ on $R$.
Indeed, $G_0^f$ is a component of the preimage of a line through $\pi_2$ but we can choose $p\in X_0$ such that it maps outside this line.
\end{proof}

Since $F_0$ is of rank $8$ and is the Fourier-Mukai kernel of a derived equivalence
$D(X_0)\simeq D(X_0')$ it induces a Hodge isometry $$\varphi\colon H^2(X_0,\mathbb{Q})\to H^2(X_0',\mathbb{Q})$$ (see \cite[Section 5.2]{Mar1}).
We consider again as in \cite[\S 5]{Mar1} a connected component of the moduli space 
$\mathcal{M}^0_{\psi}$ that
parametrises isomorphism classes of quadruples
($X_t,\eta_t,X_t',\eta_t')$ of deformation equivalent marked pairs $(X_t,\eta_t)$, where $\eta_t\colon H^2(X_t,\mathbb{Z}) \to \Lambda$ and $\eta_t'\colon H^2(X_t',\mathbb{Z}) \to \Lambda$ are
 isometries, such that 
$$\eta_1^{-1}\psi \eta_0 \colon H^2(X_0,\mathbb{Q})\to H^2(X_0',\mathbb{Q})$$
is a Hodge isometry which maps a K\"ahler class to a K\"ahler class. 
Arguing as in Section \ref{EPW2} and using Lemma \ref{VB} we can spread the bundle $F_0$ to twisted vector bundles $F_t$ on $X_t\times X_t'$.
\begin{prop}\label{convEPW} For a pair ($X_t,\eta_t,X_t',\eta_t') \in \mathcal{M}^0_{\psi}$ we have that if $X_t$ is a double EPW sextic then $X_t'$ is the dual double EPW sextic. Moreover, the bundle $F_t$ is a vector bundle and induces a derived equivalence $D(X_t)\simeq D(X_t')$.
\end{prop}
\begin{proof}
Since, $X_0$ and $X_0'$ are dual double EPW sextics we can describe the Hodge isometry $\psi$ explicitly. 
The map $\psi$ is a Hodge isometry so $$\psi(h)=ah'-bf' \ \text{and}\ \psi(e)=ch'-de'$$ such that $2 a^2-2b^2=2$ and $2c^2-2d^2=-2$ and $a,b,c,d \in \QQ$. 
By choosing appropriate markings we can assume that $\psi$ is given by an automorphism of $\Lambda$ that preserves the lattice generated by $e+f$ and $s$, is the reflection that interchanges $e-f$ and $\delta$ on the transcendental lattice (as it maps our special double EPW sextics to their dual, see \cite[Section 4.2]{O1}). Let us study how $\psi$ acts on the Picard lattice.

Now $\psi$ is the composition of two equivalences $$D(X_0)\to D(X_1)\ \text{and} \ D(X_1)\to D(X_1').$$
From \cite[\S 10.2]{B} the image through the map induced by those equivalences on second cohomologies of a point on $X_1$ is the class $t$ (resp.~$t'$) on $X_0$ (resp.~$X_0'$) where $t\in H^2(X_0,\ZZ)$ (resp.~$t'\in H^2(X_0',\ZZ)$) is the class of divisors giving the Lagrangian fibrations. If we denote by $h$ and $s$ the generators of $$\Pic(X_0)=(2)\oplus (-2)$$ such that $t=h-s$ (resp.~by $h'$ and $s'$ the generators of $\Pic(X_0')=(2)\oplus (-2)$ with $t'=h'-s'$) we infer $\psi(h-s)=h'-s'$ thus 
$$c=a-1\ \text{and}\ d=b+1.$$ This implies $a=d=1$ and $b=c=0$.
Thus we can assume (after choosing markings) that $\psi$ interchanges $e-f$ and $\delta$, maps $e+f$ to $e+f$ and is the identity elsewhere.
So it is the reflection through the class $e-f-\delta$ and we conclude from \cite[Section 4.2]{O1}) that $X_t'$ is the dual EPW sextic.


In order to show a twisted derived equivalence $D(X_t,\beta_t)\simeq D(X_t',\beta_t')$ (for some Brauer classes $\beta_t,\beta_t'$) we argue as in Theorem \ref{main markman} using Lemma \ref{VB}. To conclude we need to prove the Brauer classes are trivial.
It follows from Theorem \ref{Cardelaru} and the discussion after it that $$B=-c_1(F_0^p)/\rk(F_0^p)$$ is a $B$-lift of $\beta$. It is thus trivial since from Lemma \ref{po} we have $c_1(F_0^p)=aH-16E$ (moreover, we expect $a=4$).
\end{proof}

\section{Summary}\label{convolution}
The strategy described in the previous section can be applied in many cases.
Let us describe several twisted derived equivalences of \HK{} manifolds of Picard rank $\geq 2$.
For a K3 surface $S$ of degree $2$ there are lots of examples of moduli spaces of twisted sheaves $M_{(0,h,0)}(S,B)$ that are \HK{} fourfolds of $K3^{[2]}-$type with two Lagrangian fibrations that differ by the choice of a Brauer class $\beta \in H^2(\oo_S^{\ast})$. 
More precisely, we choose a $k$-torsion Brauer class that admits a $B$-lift of class $B$ with $$Bh=0\ \text{and}\ B^2=0$$
(we can find the $B$-lift as a sum of fractions of the generators of one of the hyperbolic summand of $\Lambda$).
For $k\geq 2$ the manifolds admit two Lagrangian fibrations which are given by $(0,0,1)$ and $(k,kB,0)$ because there are no $-2$ classes or $-10$ classes that are positive combination of the Lagrangian classes. By taking the first fibration and using Lemma \ref{relPic} we infer that the construction from Section \ref{EPW4} of relative twisted Lagrangian fibration gives a fourfold that is the moduli space $M_{(0,h,1)}(S,B)$  that we call $P_k$.
If $k$ is odd the Picard lattice is the following
$$\left( \begin{array}{cc}
 2 k^2    & 2k \\
   2k  & 0
\end{array} \right) $$
If $k$ is even it is $$\left( \begin{array}{cc}
   \frac{k^2}{2}  & \frac{k}{2} \\
     \frac{k}{2} & 0 
\end{array} \right). $$

 If we use Lemma \ref{relPic} for the second Lagrangian fibration on $P_k$ we infer a dual manifold $Q_k'$.
 As in the previous section we can prove the following.
\begin{prop} The manifolds $P_k$ and $Q_k$ are twisted derived equivalent. Moreover, $Q_k$ and $Q_k'$ are twisted derived equivalent.
\end{prop}


Let $P_{01}$ (resp.~$P_{01}'$) be a Fourier-Mukai kernel from the equivalence between $P_k$ and $Q_k$ (resp.~$Q_k$).
Let  $F^k$ be the twisted coherent sheaf obtained by the convolution of $P_{01}$ and $P'_{01}$.
The convolution is a priori more complicated along a pair of points that are the singular loci of singular fibers of abelian fibrations. We expect (cf.~Lemma \ref{VB}) that $F^k$ is, in fact, a twisted vector bundle and can be deformed (as in \cite{Mar1,O2}) as a hyperholomorphic sheaf.

The following problems arise.
\begin{prob} 
    Are all pairs of derived equivalent \HK{} manifolds of $K3^{[2]}$-type related by a composition of the constructions described in this paper?
\end{prob}
In particular are two $K3^{[2]}$ fourfolds of degree $12$ with isomorphic transcendental lattices derived equivalent?
\begin{prob}{(cf.~\cite{BB})}
    Are the equivalences constructed in Theorem \ref{main1} and Proposition \ref{convEPW} the same? What is the group of derived autoequivalences of a double EPW sextic? 
\end{prob}
It is natural to generalise the above constructions to $K3^{[n]}$-type manifolds however new technical difficulties occur when we consider convolutions in higher dimensions (cf.~\cite[\S 3]{Vo}).
In order to solve the Problem \ref{prob} for \HK{} manifolds in general it seems that new constructions of derived equivalences and auto-equivalences are needed (cf.~\cite[Remark 1.2]{T}).
Let us finally recall the following intriguing folklore problem.

\begin{prob}
 For all currently known deformation types of hyper-Kähler manifolds a derived equivalence $D(X) \simeq D(Y)$ implies that $X$ and $Y$ must, in fact, be deformation- equivalent. But in general, is it true for arbitrary hyper-Kähler manifolds?
\end{prob}

\end{document}